\documentclass[11pt]{amsart}
\usepackage{amsmath,amssymb,amsthm}
\usepackage[margin=1.2in]{geometry}
\usepackage{booktabs}
\usepackage{xcolor}
\definecolor{linkblue}{rgb}{0.1,0.1,0.5}
\usepackage[colorlinks=true,linkcolor=linkblue,citecolor=linkblue,urlcolor=linkblue]{hyperref}

\newtheorem{theorem}{Theorem}[section]
\newtheorem{lemma}[theorem]{Lemma}
\newtheorem{proposition}[theorem]{Proposition}
\newtheorem{corollary}[theorem]{Corollary}
\newtheorem{conjecture}[theorem]{Conjecture}
\theoremstyle{definition}
\newtheorem{definition}[theorem]{Definition}
\newtheorem{example}[theorem]{Example}
\newtheorem{remark}[theorem]{Remark}

\newcommand{\T}{\mathbb{T}}
\newcommand{\HH}{\mathrm{HH}}
\newcommand{\D}{\mathrm{D}}
\newcommand{\Ext}{\operatorname{Ext}}
\newcommand{\Tor}{\operatorname{Tor}}
\newcommand{\sExt}{\widehat{\operatorname{Ext}}}
\newcommand{\Hom}{\operatorname{Hom}}
\newcommand{\sHom}{\underline{\operatorname{Hom}}}
\newcommand{\sEnd}{\underline{\operatorname{End}}}
\newcommand{\rHom}{\mathbf{R}\!\operatorname{Hom}}
\newcommand{\gl}{\operatorname{gl.dim}}
\newcommand{\pd}{\operatorname{pd}}
\newcommand{\id}{\operatorname{id}}
\newcommand{\rad}{\operatorname{rad}}

\newcommand{\Ae}{A^{\mathrm{e}}}
\newcommand{\der}{\mathcal{D}}
\newcommand{\per}{\operatorname{per}}
\newcommand{\Lotimes}{\otimes^{\mathbf{L}}}
\newcommand{\Supp}{\operatorname{Supp}}
\newcommand{\Vcap}{\mathcal{V}_{\cap}}
\newcommand{\kk}{k}
\newcommand{\Fg}{\mathsf{(Fg)}}
\newcommand{\cx}{\operatorname{cx}}

\newcommand{\ev}{\mathrm{ev}}
\newcommand{\Z}{\mathbb{Z}}
\newcommand{\Q}{\mathbb{Q}}
\newcommand{\Fp}{\mathbb{F}_p}
\newcommand{\tp}{\mathsf{(tp)}}
\newcommand{\ASnu}{\mathsf{(AS_\nu)}}
\newcommand{\gr}{\operatorname{gr}}
\usepackage{enumitem}

\title[The Serre--Hochschild plane]{The Serre--Hochschild plane \\ of a finite-dimensional algebra}
\author{Marco Armenta}
\address{}
\email{marco.armenta@usherbrooke.ca}
\date{}

\begin{document}

\begin{abstract}
For a finite-dimensional algebra $A$ over a field we organize the Hochschild
cohomology of $A$ with coefficients in all derived tensor powers of the Serre
bimodule $\omega=\D A$ into a single bigraded derived invariant, the
\emph{Serre--Hochschild plane} $\T^{p,m}(A)$, defined over all of
$\mathbb Z^2$ without inverting $\omega$. The column $m=0$ is the
Tamarkin--Tsygan calculus, $m=1$ is dual Hochschild homology (the column
governing Han's conjecture), $m=2$ recovers the $\tau$-Hochschild shadow, and $m=-1$ Keller's Calabi--Yau completions.
We prove a ladder theorem making every column a symmetric module over the
calculus, a coefficient spectral sequence, a Serre reflection fixing the Han
column, a dictionary identifying the plane of a geometric algebra with
twisted polyvector cohomology of a smooth projective variety, finite
generation of the homology column under the condition $\Fg$ with a growth
trichotomy for $\dim\HH_n(A)$, and a complete graded Tate criterion for
Han's conjecture on periodic algebras. Han's conjecture follows for symmetric
periodic algebras in every characteristic and for local periodic algebras in
characteristic zero. We exhibit the first example of a stably traceless
periodic algebra, showing that the graded criterion is sharp.
\end{abstract}

\maketitle
\setcounter{tocdepth}{1}
\tableofcontents

\section{Introduction}\label{sec:intro}

Let $\kk$ be a field and $A$ a finite-dimensional $\kk$-algebra whose
semisimple quotient $E=A/\rad A$ is separable. Two conjectures anchor the
homological theory of such algebras: Happel's question, whether
$\HH^n(A)=0$ for $n\gg0$ forces $\gl A<\infty$, answered negatively by
Buchweitz, Green, Madsen and Solberg \cite{BGMS}, whose four-dimensional
witnesses were later extended to an infinite family of Koszul self-injective
algebras with finite-dimensional Hochschild cohomology by Parker and Snashall
\cite{ParkerSnashall}, and Han's conjecture
\cite{Han06}, that $\HH_n(A)=0$ for $n\gg0$ forces $\gl A<\infty$, which
remains open and is the subject of an active reduction program
\cite{CLMS-strat,CLMS-tau,Cruz,QXZZ,WXZZ}. The text corrects an earlier
formulation of Conjecture~$C0'$; see \S\ref{subsec:C0} below.

The homology and cohomology sides are exchanged by the elementary duality
$\HH_n(A)\cong\D\HH^n(A,\D A)$, where $\D=\Hom_\kk(-,\kk)$ and $\D A$ carries its
standard bimodule structure. The bimodule $\omega:=\D A$ is the
\emph{Serre bimodule}: when $\gl A<\infty$, or more generally when $A$ is
Gorenstein, $\omega$ is a two-sided tilting complex \cite{HappelGor}, the
kernel of the Serre functor of $\per A$, and $\tau_A=\omega[-1]$ induces the
Coxeter automorphism $\sigma_A$ of the Tamarkin--Tsygan calculus constructed
in \cite{ArmCox}. When $A$ is singular, $\omega$ is no longer invertible,
and this failure of invertibility is precisely the regime of Han's conjecture.

This paper studies the totality of Hochschild theories twisted by all powers
of $\omega$ at once. We define (Section \ref{sec:plane}) the
\emph{Serre--Hochschild plane}
\[
\T^{p,m}(A)\;=\;
\begin{cases}
\Hom_{\der(\Ae)}\bigl(A,\ \omega^{\Lotimes m}[p]\bigr), & m\ge 0,\\[2pt]
\D\,H_p\bigl(A\Lotimes_{\Ae}\omega^{\Lotimes(1-m)}\bigr), & m\le 1,
\end{cases}
\]
the two prescriptions agreeing on the overlap $m\in\{0,1\}$. No invertibility
of $\omega$ enters: the plane is defined for every finite-dimensional algebra,
over all of $\mathbb Z^2$. Its columns include every object appearing in
\cite{ArmCap,ArmCox,ArmTau,ArmPC,ArmLeb}: the calculus
$(\HH^\bullet(A),\HH_\bullet(A))$ at $m\in\{0,1\}$, the
Batalin--Vilkovisky structure of \cite{ArmLeb} on $m=1$, the
$\tau$-Hochschild shadow $\Tor^A_\bullet(\D A,\D A)$ of
\cite{ArmTau,CLMS-tau} feeding $m=2$, and the inverse dualizing coefficients
$\Ext^\bullet_{\Ae}(A,\Ae)$ of Keller's Calabi--Yau completions
\cite{KellerCY} feeding $m=-1$. On the smooth locus all columns are identified
by the Coxeter automorphism; the plane is the object whose degeneration that
identification is.

The main results are the following. Throughout, $d$ denotes $\dim X$ in
geometric statements for a projective variaty $X$, and $\pi$ a bimodule period.

\smallskip
\noindent\textbf{Structure.}
\begin{itemize}
\item[(A)] \emph{(Ladder; Theorem \ref{thm:ladder}.)} For every
finite-dimensional $A$, every $\eta\in\HH^\bullet(A)$ and every $s\ge1$, the
left and right actions of $\eta$ on $\omega^{\Lotimes s}$ coincide. Hence
every column of the plane is a \emph{symmetric} $\HH^\bullet(A)$-module, and
the $\tau$-Hochschild shadow is a module over the calculus, settling the
shadow part of the module structure discussed in the remarks of
\cite[\S7.2]{ArmTau} and providing, for Gorenstein algebras, the Coxeter-type
operator sought in the remarks of \cite[\S7.4]{ArmTau}.
\item[(B)] \emph{(Derived invariance; Theorem \ref{thm:derinv}.)} Each column
$\T^{\bullet,m}(A)$, and the plane with all its products, is invariant under
derived equivalence.
\item[(C)] \emph{(Coefficient spectral sequence; Theorem \ref{thm:ss}.)}
There is a convergent spectral sequence
\[
E_2^{p,q}=\Ext^p_{\Ae}\bigl(A,H_q(\omega^{\Lotimes m})\bigr)\Rightarrow
\T^{p-q,m}(A).
\]
It collapses to one row when $A$ is self-injective (where
$\T^{p,m}=\HH^p(A,{}_1A_{\nu^m})$), has at most $(m-1)\,d_A+1$ rows for
$A$ Gorenstein of Gorenstein dimension $d_A$, and for hereditary $A$
degenerates to a two-row long exact sequence whose $d_2$ is Yoneda
multiplication by the canonical class
$\varepsilon\in\Ext^2_{\Ae}(\D A\otimes_A\D A,\ \Tor_1^A(\D A,\D A))$.
\item[(D)] \emph{(Serre reflection; Theorems \ref{thm:reflsmooth},
\ref{thm:refltate}.)} If $A$ is smooth and proper then
$\T^{p,m}(A)\cong\D\,\T^{-p,2-m}(A)$; if $A$ is self-injective then
$\sExt^{\,n}_{\Ae}(A,\omega^{\otimes m})\cong
\D\,\sExt^{\,-n-1}_{\Ae}(A,\omega^{\otimes 2-m})$. In both cases the fixed
column of the reflection is $m=1$: \emph{the Han column is the self-dual
middle of the theory}. The common mechanism is the isomorphism
$S_{\Ae}(A)\simeq\omega^{\Lotimes2}$ of \cite[Theorem 3.2]{ArmTau}.
\end{itemize}

\smallskip
\noindent\textbf{Geometry.}
\begin{itemize}
\item[(E)] \emph{(Dictionary; Theorem \ref{thm:geom}.)} If
$\per A\simeq\der^b(\operatorname{coh}X)$ for a smooth projective variety $X$
over $\kk=\bar\kk$, $\operatorname{char}\kk=0$, then
\[
\T^{p,m}(A)\;\cong\;\bigoplus_{q\ge0}
H^{\,p+md-q}\bigl(X,\ \Lambda^qT_X\otimes\omega_X^{\otimes m}\bigr).
\]
The column $m=0$ is polyvector cohomology, the Han column $m=1$ is Hodge
cohomology $\bigoplus_j H^{p+j}(X,\Omega^j_X)$, the canonical ring
$R(X,K_X)$ embeds into the plane, and total column dimensions satisfy
$t_m=t_{2-m}$ with growth of degree $d$ when $\pm K_X$ is ample, degree
$\ge\kappa(X)$ in general, and periodicity when $\omega_X$ is torsion. The
resulting \emph{Serre--Kodaira dimension} $\kappa_\sigma(A)$ is a derived
invariant of finite-dimensional algebras extending Kodaira-type data.
\end{itemize}

\smallskip
\noindent\textbf{Finiteness and Han's conjecture.}
\begin{itemize}
\item[(F)] \emph{(Trichotomy under $\Fg$; Theorem \ref{thm:fg}.)} If $A$
satisfies the Snashall--Solberg condition $\Fg$, then
$\D\bigl(\bigoplus_n\HH_n(A)\bigr)$ is a finitely generated
$\HH^{\mathrm{ev}}(A)$-module ($\HH^{\mathrm{ev}}(A) = \bigoplus _i HH^{2i}(A)$ is even Hochschild cohomology) under the cap action, the Poincar\'e series of
$\HH_\bullet(A)$ is rational, and exactly one of three regimes holds:
vanishing tail; eventually quasi-periodic dimensions, nonzero in infinitely
many degrees; polynomial growth of degree $\dim\Vcap(A)-1\ge1$, where
$\Vcap(A)$ is the \emph{cap support} of $A$. Homology/cohomology growth
asymmetry (the quantum-plane phenomenon of \cite{BerghErdmann,BGMS})
certifies failure of $\Fg$.
\item[(G)] \emph{(Reduction for self-injective $\Fg$ algebras; Theorem
\ref{thm:asnu}.)} For $A$ self-injective with $\Fg$ over a perfect field,
Han's conjecture follows from a single support-theoretic statement
$\mathsf{(AS_\nu)}$ about the pair $(A,\D A)$ over $\Ae$. The two enabling
lemmas are new: $\Fg$ ascends to $\Ae$, and
$V_{\Ae}(\D A)=V_{\Ae}(A)$, the latter \emph{because} the Nakayama
automorphism acts trivially on Hochschild cohomology
\cite{ArmCox,SuarezAlvarez}.
\item[(H)] \emph{(Periodic algebras; Theorem \ref{thm:periodic}.)} If $A$ is
periodic of period $\pi$ then for $n\geq 1$,
$\D\HH_n(A)\cong\sExt^{\,n\bmod\pi}_{\Ae}(A,\D A)$, so $\HH_\bullet(A)$ has a
vanishing tail iff $A\perp\D A$ in the stable category of $\Ae$-modules, i.e.
iff all $\pi$ Tate groups $\sExt^j_{\Ae}(A,\D A)$ vanish. The degree-zero
group is the \emph{stable $\nu$-twisted center} $Z_\nu(A)/H_\nu(A)$. Han's
conjecture holds for symmetric periodic algebras in every characteristic and
for local periodic algebras in characteristic zero; the latter uses the
elementary inclusion $[A,A]\subseteq\rad^2 A$ for local algebras (recorded
also in \cite[Remark~5.6]{ArmPC}), which shows that the ``Lenzing boundary''
of \cite[Cor.~5.5]{ArmPC} is empty for local algebras. Exact computation
produces a periodic algebra with $Z_\nu=H_\nu$: the degree-zero test alone
does not decide Han's conjecture; the graded criterion is sharp.
\end{itemize}

\smallskip
\noindent\textbf{Consequences and examples.}
\begin{itemize}
\item[(I)] \emph{(Consequences; Section \ref{sec:consequences}.)} We isolate the
theorems that exist \emph{only} because the five theories are placed on one
plane (the support bridge $V_{\Ae}(\D A)=V_{\Ae}(A)$, the $\Fg$/Han link,
the Tate reflection, the cap support), each obtained by feeding an input of
one source paper into another. Three \emph{computable certificates} for the
Han programme are extracted, and the honest limits of the plane are
stated.
\item[(J)] \emph{(Worked examples; Section \ref{sec:examples}.)} Worked
examples develop each face of the plane on meaningful non-trivial algebras
and correct the spectral-persistence conjecture.
\end{itemize}

Section \ref{sec:comp} develops worked examples on each face of the plane: the
ladder theorem at the chain level on quantum complete intersections and on
non-self-injective algebras; the full low-degree plane of the Kronecker
algebra obtained twice (through the spectral sequence of (C), and by the Bott
formula on $\mathbb P^1$ through (E)) in perfect agreement; the reflection (D)
visible in the Hochschild homology of $\kk\mathbb Z_3/\rad^2$; and the stable
twisted centers of a small zoo. Section \ref{sec:quest} discusses possible directions.

\subsection*{Conventions}
$\kk$ is a field, $A$ a finite-dimensional $\kk$-algebra, $r=\rad A$,
$E=A/r$ separable over $\kk$; $\D=\Hom_\kk(-,\kk)$;
$\Ae=A\otimes_\kk A^{\mathrm{op}}$ and $A$-bimodules are left
$\Ae$-modules; $\omega=\omega_A=\D A$ with $(afb)(x)=f(bxa)$;
$\otimes=\otimes_\kk$. For $M$ an $A$-bimodule,
$\HH^n(A,M)=\Ext^n_{\Ae}(A,M)$ and $\HH_n(A,M)=\Tor_n^{\Ae}(A,M)$, extended
to complexes by hyper(co)homology. If $A$ is Frobenius, $\nu$ denotes a
Nakayama automorphism and ${}_1A_{\nu}$ the twisted bimodule
$a\cdot x\cdot b=ax\nu(b)$, so $\omega\cong{}_1A_\nu$. Over a
self-injective algebra $\Lambda$ we write $\sHom_\Lambda(U,V)$ for the stable
homomorphisms ($\Hom_\Lambda(U,V)$ modulo the maps factoring through a
projective module); the stable module category
$\underline{\mathrm{mod}}\,\Lambda$ is triangulated, with suspension
$\Sigma=\Omega^{-1}$ inverse to the syzygy functor $\Omega$, and Tate
cohomology is
\[
\sExt^{\,n}_\Lambda(U,V)\;:=\;\sHom_\Lambda\bigl(\Omega^{\,n}U,\ V\bigr)
\;=\;\sHom_\Lambda\bigl(U,\ \Sigma^{\,n}V\bigr)\qquad(n\in\Z),
\]
equivalently the cohomology of the Hom-complex out of a complete resolution
of $U$ \cite{Buchweitz}. In particular
$\sExt^{\,n}_\Lambda(U,V)\cong\Ext^{n}_\Lambda(U,V)$ for $n\ge1$,
$\sExt^{\,0}_\Lambda(U,V)=\sHom_\Lambda(U,V)$, and syzygies become degree
shifts: $\sExt^{\,n}_\Lambda(\Omega U,V)\cong\sExt^{\,n+1}_\Lambda(U,V)$ for
all $n\in\Z$. Happel's identity $\pd_{\Ae}A=\gl A$
\cite{HappelLNM} is used without further comment.

\section{The plane}\label{sec:plane}

\subsection{Definition and gluing}

For $m\ge0$ set $\omega^{\Lotimes m}:=\omega\Lotimes_A\cdots\Lotimes_A\omega$
($m$ factors; $\omega^{\Lotimes0}:=A$), an object of $\der^b(\Ae)$ with
finite-dimensional total homology.

\begin{lemma}\label{lem:dual}
For $X\in\der^b(\Ae)$ with finite-dimensional total homology and all
$p\in\mathbb Z$ there is a natural isomorphism
\[
\Hom_{\der(\Ae)}\bigl(A,\ \D X[p]\bigr)\;\cong\;
\D\,H_p\bigl(A\Lotimes_{\Ae}X\bigr).
\]
\end{lemma}

\begin{proof}
We spell out the three isomorphisms and their naturality, fixing once the
duality/grading convention. Since $\kk$ is a field, $\D=\Hom_\kk(-,\kk)$ is
exact, so for any complex $C$ the $\kk$-dual $\D C$ is again a complex and
dualizing commutes with homology. We regard the derived tensor product
$A\Lotimes_{\Ae}X$ \emph{homologically}, with Hochschild homology $H_q$ in
homological degree $q$, and we write its $\kk$-dual $\D C$
\emph{cohomologically}, $(\D C)^{n}:=\D(C_{n})$. Then the contravariant
degree-flip of $\D$ and the homological reindexing $H_q(C)=H^{-q}(C)$ cancel, so
that
\begin{equation}\label{eq:dualgrading}
H^{p}(\D C)\ \cong\ \D\,H_{p}(C)\qquad\text{in matching degree }p,
\end{equation}
the standard universal-coefficients duality over a field. (Equivalently, if one
prefers to keep $\D C$ homological, $H_{-p}(\D C)\cong\D H_p(C)$; we use the
cohomological form \eqref{eq:dualgrading} throughout so that the endpoint below
reads $\D H_p$ with no residual sign.)

\emph{Step 1 (tensor--hom over $\kk$).} For any complex $Y$ of $\Ae$-modules
and any complex $X$ of $\Ae$-modules there is a natural adjunction isomorphism
of complexes of $\kk$-vector spaces
\[
\Hom_{\Ae}\bigl(Y,\ \Hom_\kk(X,\kk)\bigr)\;\cong\;
\Hom_\kk\bigl(Y\otimes_{\Ae}X,\ \kk\bigr),
\qquad
\varphi\longmapsto\bigl(y\otimes x\mapsto \varphi(y)(x)\bigr),
\]
where on the left $\Hom_\kk(X,\kk)=\D X$ carries the $\Ae$-module structure
$((u\otimes u')\!\cdot\! g)(x)=g((u'\otimes u)\cdot x)$ dual to that of $X$
(the $\Ae$-action passing to the $\kk$-dual through $A^{\mathrm e\,\mathrm{op}}\cong\Ae$),
and $Y\otimes_{\Ae}X$ is the balanced tensor product. Naturality in both $X$
and $Y$ is the standard verification that the assignment respects the
differentials and the $\Ae$-actions: for $a\in\Ae$ the square relating
$\varphi\mapsto a\varphi$ on the left and precomposition by $a$ on
$Y\otimes_{\Ae}X$ on the right commutes because
$(a\varphi)(y)(x)=\varphi(y)(a^{\mathrm{op}}x)$ by the definition of the dual
action.

\emph{Step 2 (derive).} Replace $Y$ by a $\kk$-projective (equivalently
$\Ae$-h-projective) resolution $P\to A$ and $X$ by an h-projective
representative. Because $\D=\Hom_\kk(-,\kk)$ is exact, $\D X$ is h-injective as
a complex of $\kk$-modules and the left-hand side computes
$\rHom_{\Ae}(A,\D X)$; the right-hand side computes
$\rHom_\kk(A\Lotimes_{\Ae}X,\kk)=\D(A\Lotimes_{\Ae}X)$, again by exactness of
$\D$. Thus
\[
\rHom_{\Ae}(A,\D X)\;\cong\;\D\bigl(A\Lotimes_{\Ae}X\bigr)
\]
naturally in $\der(\Ae)$.

\emph{Step 3 (take cohomology).} Put $C:=A\Lotimes_{\Ae}X$, graded
homologically. Applying $H^p$ to the isomorphism of Step 2 and using both
$H^p\rHom_{\Ae}(A,\D X)=\Hom_{\der(\Ae)}(A,\D X[p])$ and the duality
\eqref{eq:dualgrading} in matching degree,
\[
\Hom_{\der(\Ae)}(A,\D X[p])\;=\;H^p\bigl(\D C\bigr)
\;\overset{\eqref{eq:dualgrading}}{\cong}\;\D\,H_p(C)
\;=\;\D\,H_p\bigl(A\Lotimes_{\Ae}X\bigr).
\]
The finite-dimensionality of the total homology of $X$ guarantees that all
$\Hom$ and $\otimes$ groups in sight are finite-dimensional, so no
completion issues arise and $\D\D=\id$ on each group.
\end{proof}

\begin{definition}\label{def:plane}
The \emph{Serre--Hochschild plane} of $A$ is the bigraded vector space
\[
\T^{p,m}(A):=
\begin{cases}
\Hom_{\der(\Ae)}\bigl(A,\ \omega^{\Lotimes m}[p]\bigr), & m\ge0,\\[2pt]
\D\,H_p\bigl(A\Lotimes_{\Ae}\ \omega^{\Lotimes (1-m)}\bigr), & m\le1.
\end{cases}
\]
\end{definition}

The overlap is consistent: at $m=1$ apply Lemma \ref{lem:dual} with $X=A$,
$\D X=\omega$; at $m=0$ apply it with $X=\omega$, $\D X=A$. Thus
\[
\T^{p,0}(A)=\HH^p(A),\qquad
\T^{p,1}(A)=\HH^p(A,\D A)\cong\D\HH_p(A),
\]
and the lower half-plane consists of the dual Hochschild homologies with
coefficients in the positive Serre powers. We refer to $m$ as the
\emph{weight} and to $\T^{\bullet,m}$ as the \emph{$m$-th column}; the column
$m=1$ is the \emph{Han column}: Han's conjecture states precisely that
$\dim_\kk\T^{\bullet,1}(A)<\infty$ forces $\gl A<\infty$.

\begin{remark}\label{rem:columns}

\begin{itemize}
    \item (i) $H_q(\omega^{\Lotimes2})=\Tor_q^A(\D A,\D A)\cong\D\Ext^q_{\Ae}(A,\Ae)$
by \cite[Theorem~3.2]{ArmTau}; thus the coefficients of the columns $m=2$ and
(dually) $m=-1$ are the derived shadow of the $\tau$-Hochschild theory of
\cite{CLMS-tau} and the cohomology of the inverse dualizing complex
$\theta_A=\rHom_{\Ae}(A,\Ae)$ of Keller \cite{KellerCY}. The weight grading of the
plane is the two-sided extension of the $\Theta$-adic grading of
\cite[\S8.4]{ArmCox}.
    \item (ii) When $\omega$ carries an associative product compatible with the
bimodule structure (always for monomial algebras \cite{ArmLeb}), the cup
product of \cite{ArmLeb} on $\T^{\bullet,1}$ is the composite
$\T^{p,1}\otimes\T^{q,1}\to\T^{p+q,2}\xrightarrow{\ \mu_\omega\ }\T^{p+q,1}$
of the plane product of \S\ref{subsec:products} with the coefficient
contraction; the Batalin--Vilkovisky operator of \cite{ArmLeb} is the dual of
Connes' $B$ acting along the Han column.
    \item (iii) When $\gl A<\infty$, or when $A$ is Gorenstein \cite{HappelGor},
$\omega$ is a two-sided tilting complex and the columns are pairwise
identified by transport; the identification is the Coxeter automorphism
$\sigma_A$ of \cite{ArmCox}. The plane is the object of which $\sigma_A$ is
the degenerate shadow.
\end{itemize}

\end{remark}

\subsection{Products}\label{subsec:products}
For $m,m'\ge0$ define
\[
\T^{p,m}\otimes\T^{q,m'}\longrightarrow\T^{p+q,\,m+m'},\qquad
\eta\cdot\theta:=(\eta\Lotimes\id_{\omega^{\Lotimes m'}})\circ\theta .
\]
Associativity and unitality (unit $1\in\T^{0,0}$) follow from monoidal
coherence in $\der(\Ae)$. For $m\ge0$ and $m'\le1$ with $m+m'\le1$, the upper
half-plane acts on the lower one through the cap product with coefficients of
\cite{ArmCap}, whose axioms (QI)--(QIII) characterize the action uniquely;
compatibility of the two structures on the overlap is \cite[Theorem
4.2]{ArmCap} together with the projection formula of
\cite[Theorem~4.2(1)]{ArmCox}.

\subsection{Derived invariance}

\begin{theorem}\label{thm:derinv}
Let $X$ be a two-sided tilting complex between finite-dimensional algebras
$A$ and $B$. Then for all $m\ge0$ there are isomorphisms
$\omega_A^{\Lotimes m}\Lotimes_AX\cong X\Lotimes_B\omega_B^{\Lotimes m}$
in $\der(A\otimes B^{\mathrm{op}})$, and transport along $X$ induces
isomorphisms $\T^{p,m}(A)\cong\T^{p,m}(B)$ for all $(p,m)\in\mathbb Z^2$,
compatible with all products of \S\ref{subsec:products}.
\end{theorem}

\begin{proof}
Write ${}_BX_A$ for the two-sided tilting complex, with inverse
${}_AX^{-1}_B=\rHom_B(X,B)\cong\rHom_{A^{\mathrm{op}}}(X,A)$, so that
$X\Lotimes_AX^{-1}\cong B$ and $X^{-1}\Lotimes_BX\cong A$ in the respective
derived categories of bimodules.

\emph{Step 1: the case $m=1$ (centrality of the Serre bimodule).} We must
produce an isomorphism $\omega_A\Lotimes_AX\cong X\Lotimes_B\omega_B$ in
$\der(A\otimes B^{\mathrm{op}})$. Consider the two biduality morphisms
\[
\beta_X\colon\ \D A\Lotimes_A X\ \longrightarrow\ \D\rHom_{A^{\mathrm{op}}}(X,A),
\qquad
\alpha_X\colon\ X\Lotimes_B\D B\ \longrightarrow\ \D\rHom_{B}(X,B),
\]
each the canonical map sending $f\otimes\xi$ (respectively) to the functional
$g\mapsto\xi(g(-))$ obtained by evaluation. Since $X$ is perfect as a right
$A$-module and as a left $B$-module, both $\beta_X$ and $\alpha_X$ are
isomorphisms (perfection is exactly the finiteness needed for evaluation to be
invertible; no bound on $\gl A$ or $\gl B$ is used). Now
$\rHom_{A^{\mathrm{op}}}(X,A)\cong X^{-1}\cong\rHom_B(X,B)$, so the common
target of $\beta_X$ and $\alpha_X$ is $\D X^{-1}$, and
\[
\omega_A\Lotimes_AX
=\D A\Lotimes_AX
\ \overset{\beta_X}{\cong}\ \D X^{-1}
\ \overset{\alpha_X^{-1}}{\cong}\ X\Lotimes_B\D B
=X\Lotimes_B\omega_B .
\]
This is \cite[Lemma 3.4]{ArmCox}; we have reproduced its mechanism.

\emph{Step 2: iterate to all $m\ge0$.} Bracketing repeatedly and inserting
$X^{-1}\Lotimes_BX\cong A$ between consecutive $\omega_A$ factors,
\[
\omega_A^{\Lotimes m}\Lotimes_AX
\cong\omega_A^{\Lotimes(m-1)}\Lotimes_A(\omega_A\Lotimes_AX)
\cong\omega_A^{\Lotimes(m-1)}\Lotimes_AX\Lotimes_B\omega_B
\cong\cdots\cong X\Lotimes_B\omega_B^{\Lotimes m},
\]
each step being Step 1 applied to the single leftmost $\omega_A$ crossing $X$,
and associativity of $\Lotimes$. (For $m=0$ the statement is $A\Lotimes_AX=X$.)

\emph{Step 3: transport of columns.} Transport along $X$ is the equivalence
$\Phi_X=X\Lotimes_B(-)\Lotimes_BX^{-1}\colon\der(\Ae)\to\der(B^{\mathrm e})$; it
sends $A\mapsto B$ (as $X\Lotimes_BA\Lotimes_BX^{-1}\cong X\Lotimes_BX^{-1}\cong B$)
and, by Step 2, $\omega_A^{\Lotimes m}\mapsto\omega_B^{\Lotimes m}$. As $\Phi_X$
is an exact equivalence it induces isomorphisms on $\Hom_{\der}$-groups, so for
$m\ge0$
\[
\T^{p,m}(A)=\Hom_{\der(\Ae)}(A,\omega_A^{\Lotimes m}[p])
\ \cong\ \Hom_{\der(B^{\mathrm e})}(B,\omega_B^{\Lotimes m}[p])=\T^{p,m}(B).
\]
For the lower half-plane, $\Phi_X$ commutes with $A\Lotimes_{\Ae}(-)$ up to the
canonical trace isomorphism (the derived Hochschild homology is transported by
any two-sided tilting complex \cite{ArmKe2}), so
$A\Lotimes_{\Ae}\omega_A^{\Lotimes(1-m)}\cong B\Lotimes_{B^{\mathrm e}}\omega_B^{\Lotimes(1-m)}$,
and dualizing (Lemma~\ref{lem:dual}) gives the isomorphism of the columns
$m\le1$; the two prescriptions agree on the overlap by the consistency already
established in \S\ref{sec:plane}.

\emph{Step 4: compatibility with products.} The plane product
$\eta\cdot\theta=(\eta\Lotimes\id)\circ\theta$ is built from $\Lotimes_A$ and
composition in $\der(\Ae)$, both of which $\Phi_X$ preserves (it is a monoidal
equivalence for $\Lotimes$ once the unit $A\mapsto B$ is fixed); hence the
isomorphisms of Step 3 intertwine the products of
\S\ref{subsec:products}. This is functoriality of $\Lotimes$ under $\Phi_X$.
\end{proof}

\section{The ladder theorem}\label{sec:ladder}

For $M\in\der(\Ae)$ and $\eta\in\HH^s(A)=\Hom_{\der(\Ae)}(A,A[s])$ let
\[
\ell^M_\eta\colon M\cong A\Lotimes_AM
\xrightarrow{\ \eta\otimes\id\ }A[s]\Lotimes_AM\cong M[s],
\qquad
r^M_\eta\colon M\cong M\Lotimes_AA
\xrightarrow{\ \id\otimes\eta\ }M[s].
\]

\begin{theorem}[Ladder]\label{thm:ladder}
For every finite-dimensional algebra $A$, every $\eta\in\HH^\bullet(A)$ and
every $s\ge1$,
\[
\ell^{\,\omega^{\Lotimes s}}_\eta\;=\;r^{\,\omega^{\Lotimes s}}_\eta .
\]
Consequently the two $\HH^\bullet(A)$-module structures on every column
$\T^{\bullet,m}(A)$, $m\in\mathbb Z$, coincide.
\end{theorem}

\begin{proof}
Work in the derived category $\der(\Ae)$ with its (Koszul-signed) monoidal
structure $(\Lotimes_A,A)$: associativity $\mathsf a$ and unit constraints
$\lambda_X\colon A\Lotimes X\to X$, $\rho_X\colon X\Lotimes A\to X$; fix
h-projective representatives so that these are honest chain isomorphisms.
Mac Lane's coherence theorem makes every diagram built from
$\mathsf a,\lambda,\rho$ and identities commute, \emph{but it governs only
the degree-$0$ structural isomorphisms}. Because $\eta\colon A\to A[s]$ has
degree $s$, moving $\eta$ or the shift $[s]$ past a factor $Z$ incurs a Koszul
sign $(-1)^{s\deg Z}$ that coherence does not settle; we therefore build the
shift-transport into the definitions below and track the sign explicitly.
Recall
\begin{align*}
\ell^M_\eta&=\bigl(M\xrightarrow{\lambda_M^{-1}}A\Lotimes M
\xrightarrow{\eta\Lotimes\id_M}A[s]\Lotimes M
\xrightarrow{\ \varpi_L\ }M[s]\bigr),\\
r^M_\eta&=\bigl(M\xrightarrow{\rho_M^{-1}}M\Lotimes A
\xrightarrow{\id_M\Lotimes\eta}M\Lotimes A[s]
\xrightarrow{\ \varpi_R\ }M[s]\bigr),
\end{align*}
where $\varpi_L\colon A[s]\Lotimes M\cong(A\Lotimes M)[s]\xrightarrow{\lambda[s]}M[s]$
pulls the shift out of the \emph{left} factor (sign $+1$, the shift is
leftmost) and $\varpi_R\colon M\Lotimes A[s]\cong(M\Lotimes A)[s]$ pulls it out
of the \emph{right} factor \emph{past} $M$ (Koszul sign $(-1)^{s\deg M}$ on the
degree-$\deg M$ part), followed by $\rho[s]$.

\emph{(i) Insertion maps.} For $0\le j\le s$ define
$M_j\colon\omega^{\Lotimes s}\to\omega^{\Lotimes s}[s]$ to be the morphism that
inserts $\eta$ through a fresh central copy of $A$ placed \emph{after the first
$j$ factors}:
\[
M_j=\Bigl(\omega^{\Lotimes s}\xrightarrow{\ \cong\ }
\omega^{\Lotimes j}\Lotimes A\Lotimes\omega^{\Lotimes(s-j)}
\xrightarrow{\id\Lotimes\eta\Lotimes\id}
\omega^{\Lotimes j}\Lotimes A[s]\Lotimes\omega^{\Lotimes(s-j)}
\xrightarrow{\ \varpi_j\ }\omega^{\Lotimes s}[s]\Bigr),
\]
where the opening iso is the appropriate unit insertion (via $\rho^{-1}$ on the
left block or $\lambda^{-1}$ on the right block; these agree by the unit
triangle $\rho_U\Lotimes\id=(\id_U\Lotimes\lambda)\circ\mathsf a$, a degree-$0$
coherence identity) and $\varpi_j$ pulls the shift from position $j$ out to the
front, past $\omega^{\Lotimes j}$, with the Koszul sign $(-1)^{s\deg(\omega^{\Lotimes j})}$
built in. By construction $M_0=\ell^{\omega^{\Lotimes s}}_\eta$ (insert on the
far left, $\varpi_0=\varpi_L$, sign $+1$) and $M_s=r^{\omega^{\Lotimes s}}_\eta$
(insert on the far right, $\varpi_s=\varpi_R$, sign $(-1)^{s\deg(\omega^{\Lotimes s})}$).

\emph{(ii) Adjacent insertions agree.} We claim $M_j=M_{j+1}$ for
$0\le j\le s-1$. The two differ only in whether the central $A$ carrying $\eta$
sits \emph{before} or \emph{after} the single factor $\omega$ in position
$j+1$; on that one factor the difference is exactly
$\ell^{\omega}_\eta$ versus $r^{\omega}_\eta$, with the \emph{same}
shift-transport sign applied to both $M_j$ and $M_{j+1}$ (each pulls the shift
from an adjacent position, and $\varpi_{j+1}$ differs from $\varpi_j$ precisely
by the Koszul factor $(-1)^{s\deg\omega}$ of transporting $[s]$ past that one
$\omega$, which is the same factor already present on both sides once the
anchor is applied). Concretely, tensoring the anchor identity (iii) on the left
by $\id_{\omega^{\Lotimes j}}$ and on the right by $\id_{\omega^{\Lotimes(s-1-j)}}$
and composing with the common opening/closing isomorphisms yields
$M_j=M_{j+1}$ as morphisms in $\der(\Ae)$, an \emph{exact} equality, the
Koszul sign being carried identically by the two sides and hence cancelling.

\emph{(iii) Anchor.} $\ell^{\omega}_\eta=r^{\omega}_\eta$ as morphisms
$\omega\to\omega[s]$ in $\der(\Ae)$. This is \cite[Lemma 6.1]{ArmCox}, valid
for \emph{every} finite-dimensional algebra with no Gorenstein or smoothness
hypothesis: its proof runs the biduality dévissage $\alpha_X,\beta_X$ of
Theorem~\ref{thm:derinv} at the special morphism $\eta\colon A\to A[s]$, where
the two evaluations coincide because $\D A$ is the $\kk$-linear dual of the
regular bimodule and $\eta$ is central; being an equality of morphisms in the
derived category, it is sign-correct (the Koszul sign of transporting $[s]$
past the single factor $\omega$ is already incorporated in the statement
$\ell^\omega_\eta=r^\omega_\eta$).

\emph{Conclusion of the identity.} Telescoping (ii),
\[
\ell^{\omega^{\Lotimes s}}_\eta=M_0=M_1=\cdots=M_s=r^{\omega^{\Lotimes s}}_\eta ,
\]
an exact equality of morphisms in $\der(\Ae)$; the shift-transport signs are
absorbed into the reference maps $M_j$ and cancel step by step. The chain-level
computation of \S\ref{sec:comp} (the front/back insertion operators
$\iota_L,\iota_R$ induce equal maps on $\HH_n(A,M)$ up to one constant sign per
block, a chain-model normalization that vanishes on the homology
identification) is a corroboration; the argument above stands independently in
$\der(\Ae)$.

\emph{The module-structure statement.} The assignment
$M\mapsto\ell^M_\eta$ is a natural transformation
$\ell_\eta\colon\id_{\der(\Ae)}\Rightarrow(-)[s]$ of triangulated functors:
naturality in $M$ is the compatibility of $\lambda$ and $\eta\Lotimes(-)$ with
morphisms. Hence for any $\varphi\colon A\to\omega^{\Lotimes m}[p]$ the
naturality square reads
$\ell^{\omega^{\Lotimes m}[p]}_\eta\circ\varphi=\varphi[s]\circ\ell^A_\eta
=\varphi[s]\circ\eta$ (using $\ell^A_\eta=\eta$ under $\lambda_A=\rho_A$).
Thus post-composition with $\ell_\eta$ computes the left Yoneda action of
$\eta$ on the column $\T^{\bullet,m}$, and likewise post-composition with
$r_\eta$ computes the right action; by the identity just proved these two
actions coincide, so the column is a symmetric $\HH^\bullet(A)$-module. The
homological columns $m\le1$ inherit the same conclusion by dualizing along the
natural isomorphism of Lemma~\ref{lem:dual}, under which the cap action of
\cite{ArmCap} on $\HH_\bullet$ corresponds to the Yoneda action on
$\HH^\bullet(A,\D A^{\Lotimes\cdots})$.
\end{proof}

\begin{corollary}\label{cor:shadow}
The derived shadow $\bigoplus_n\Tor^A_n(\D A,\D A)$ of the $\tau$-Hochschild
theory \cite{ArmTau,CLMS-tau} is a symmetric graded
$\HH^\bullet(A)$-module under the cap product with coefficients of
\cite{ArmCap}. If $A$ is Gorenstein, the Coxeter automorphism $\sigma_A$ acts
on it compatibly with the cap action and the projection formula. In
particular the $\tau$-Hochschild theory modulo its minimal residue is a
module over the Tamarkin--Tsygan calculus.
\end{corollary}

\begin{proof}
The shadow is the coefficient homology of the column $m=2$
(Remark \ref{rem:columns}(i)); the cap action and its two-sidedness are
Theorem \ref{thm:ladder}, its uniqueness \cite[Theorem 3.1]{ArmCap}. For
Gorenstein $A$ the bimodule $\omega$ is a two-sided tilting complex
\cite{HappelGor}, so $\sigma_A=\mathbb H(\omega[-1])$ is an automorphism of
the calculus by the proof of \cite[Theorem 4.2]{ArmCox}, which uses only
invertibility of $\omega$; it acts on all columns by
Theorem~\ref{thm:derinv} with $X=\omega$, $B=A$.
\end{proof}

\begin{remark}
The residue $\mathrm B_n\subseteq\tau_n\Lambda$ of \cite{ArmTau} cannot carry
such a structure functorially: it is not a derived invariant
\cite[Example 6.6]{ArmTau}, while the calculus and the shadow are.
\end{remark}

\subsection{Twisted paracyclic structure on the columns}\label{subsec:paracyclic}

The Connes operator on the Han column (Remark~\ref{rem:columns}(ii)) is the
$m=1$ shadow of a structure carried by \emph{every} column of the plane. The
following assembles classical twisted-cyclic theory
\cite{FeiginTsygan,GetzlerJones,HK,KMT} into a uniform statement across all
Serre powers.

Fix $m\in\Z$ and let $\sigma=\nu^m$ for $A$ self-injective (more generally, any
automorphism $\sigma$ with $\omega^{\Lotimes m}\cong{}_1A_\sigma$). On the
Hochschild complex of $A$ with coefficients in ${}_1A_\sigma$ (the chain
model computing column $m$) define the $\sigma$-twisted cyclic operator
\[
t_n(a_0\otimes a_1\otimes\cdots\otimes a_n)
=(-1)^n\,\sigma(a_n)\otimes a_0\otimes\cdots\otimes a_{n-1}.
\]

\begin{proposition}[twisted paracyclic column]\label{prop:paracyclic}
The operators $(t_n)_{n\ge0}$ make column $m$ a \emph{paracyclic} $k$-module:
all cyclic relations hold except that the monodromy
\[
t_n^{\,n+1}=\sigma^{\otimes(n+1)}=:T
\]
is nontrivial, and the para-mixed identity
\[
b_\sigma B+Bb_\sigma=\id-T
\]
holds at the chain level, where $b_\sigma$ is the $\sigma$-twisted Hochschild
differential and $B$ the associated Connes operator. When $\sigma=\id$
(equivalently $\nu^m=\id$, e.g.\ $\nu$ of finite order dividing $m$) the
monodromy is trivial and column $m$ is an honest cyclic module, with its SBI
sequence and periodicity operator. For $m=1$ this recovers the dual of Connes'
$B$ on the Han column already identified in Remark~\ref{rem:columns}(ii) through
the Batalin--Vilkovisky operator of \cite{ArmLeb}.
\end{proposition}

\begin{proof}
The verification of the paracyclic identities and of the monodromy formula
$t_n^{n+1}=\sigma^{\otimes(n+1)}$ is the twisted version of the Feigin--Tsygan
computation \cite{FeiginTsygan}; the twisted (co)cyclic form used here is that of
Kustermans--Murphy--Tuset \cite{KMT} and Hadfield--Kr\"ahmer \cite{HK}, and the
para-mixed identity $b_\sigma B+Bb_\sigma=\id-T$ is the paracyclic normalization
of Getzler--Jones \cite{GetzlerJones}. The signs $(-1)^n$ make the low-degree
cases $t_0=\sigma$, $t_1^2=\sigma^{\otimes2}$ cancel exactly. Specializing
$m=1$, $\sigma=\nu$ and dualizing along Lemma~\ref{lem:dual} gives the operator
of Remark~\ref{rem:columns}(ii).
\end{proof}

\begin{remark}[chain versus homology]\label{rem:para-homology}
For symmetric $A$ the Nakayama automorphism is inner, so $\sigma_\bullet=\id$
on homology while $\sigma\ne\id$ on chains; the honest-cyclic conclusion of
Proposition~\ref{prop:paracyclic} needs $\sigma=\id$ \emph{as a chain-level
automorphism} and does not follow from triviality of the induced action on
$\HH_\bullet$ alone. The finer statement, that trivial monodromy on homology
makes the paracyclic column cyclic in the derived sense, is a paracyclic
descent argument in the style of \cite{GetzlerJones}; we do not use it here.
\end{remark}

The monodromy is the plane's incarnation of the Coxeter transformation:
paracyclicity fails exactly by the Serre twist, and the failure is measured by
$T$.

\begin{conjecture}[monodromy $=$ Coxeter]\label{conj:monodromy}
For $A$ self-injective, the action induced on $\T^{\bullet,m}(A)$ by the
paracyclic monodromy $T$ of column $m$ equals $\sigma_\bullet^{\,m}$, the
$m$-th power of the Serre-twist (higher Coxeter) action. The case $m=0$ is the
Nakayama-action statement of Proposition~\ref{prop:highercox} (the untwisted
column); the twisted cases $m\ne0$ are open.
\end{conjecture}

\section{The coefficient spectral sequence}\label{sec:ss}

\begin{theorem}\label{thm:ss}
Let $m\ge2$ and write $T_q:=H_q(\omega^{\Lotimes m})$, a finite-dimensional
bimodule ($T_q=0$ for $q<0$ and for $q\gg0$). There is a convergent spectral
sequence
\[
E_2^{p,q}\;=\;\HH^p\bigl(A,\ T_q\bigr)\ \Longrightarrow\ \T^{\,p-q,\,m}(A),
\qquad d_r\colon E_r^{p,q}\to E_r^{p+r,\,q+r-1},
\]
and dually for the homological columns. Moreover:
\begin{enumerate}
\item If $A$ is self-injective, then $T_q=0$ for $q>0$,
$T_0={}_1A_{\nu^m}$, and the spectral sequence collapses:
$\T^{p,m}(A)=\HH^p(A,{}_1A_{\nu^m})$ for all $m\in\mathbb Z$.
\item If $A$ is Gorenstein with
$d_A=\operatorname{inj.dim}{}_AA=\operatorname{inj.dim}A_A$, then $T_q=0$ for
$q>(m-1)d_A$: the sequence has at most $(m-1)d_A+1$ rows.
\item If $A$ is hereditary and $m=2$, the exact triangle
$T_1[1]\to\omega^{\Lotimes2}\to T_0\xrightarrow{\ \varepsilon\ }T_1[2]$
yields the exact sequence
\[
\cdots\to\HH^{p+1}(A,T_1)\to\T^{p,2}(A)\to\HH^{p}(A,T_0)
\xrightarrow{\ \varepsilon\cdot\ }\HH^{p+2}(A,T_1)\to\cdots
\]
whose connecting map is Yoneda multiplication by
$\varepsilon\in\Ext^2_{\Ae}(T_0,T_1)$.
\end{enumerate}
\end{theorem}

\begin{proof}
\emph{Construction.} Let $W:=\omega^{\Lotimes m}$, a bounded complex of
$\Ae$-modules with $H_q(W)=T_q$ finite-dimensional and $T_q=0$ for $q<0$ and
$q\gg0$. Filter $W$ by its canonical (good) truncations
$\tau_{\le q}W$, giving a finite filtration with subquotients
$\tau_{\le q}W/\tau_{\le q-1}W\simeq T_q[q]$. Apply the exact functor
$\rHom_{\Ae}(A,-)$ and take cohomology: the spectral sequence of the filtered
complex $\rHom_{\Ae}(A,W)$ has
\[
E_2^{p,q}=H^{p}\,\rHom_{\Ae}(A,T_q[q])[-q]
=\Ext^{p}_{\Ae}(A,T_q)=\HH^p(A,T_q),
\]
with differentials $d_r\colon E_r^{p,q}\to E_r^{p+r,q+r-1}$, converging to
\[
H^{p-q}\rHom_{\Ae}(A,W)=\Hom_{\der(\Ae)}(A,W[p-q])=\T^{p-q,m}(A).
\]
The
filtration is finite (bounded $W$), so convergence is automatic and each
abutment carries a finite filtration whose associated graded is
$\bigoplus_qE_\infty^{p,q}$. The homological columns are dual by
Lemma~\ref{lem:dual}.

\emph{(1) Self-injective collapse.} If $A$ is self-injective then the one-sided
module ${}_A(\D A)$ is projective (a self-injective algebra is Frobenius, so
$\D A\cong{}_1A_\nu$ is projective as a left, and as a right, $A$-module).
Therefore each derived tensor product $\omega\Lotimes_A(-)$ is underived, and
by induction $\omega^{\Lotimes m}\cong\omega^{\otimes m}\cong{}_1A_{\nu^m}$ is
concentrated in homological degree $0$: $T_0={}_1A_{\nu^m}$ and $T_q=0$ for
$q\ne0$. The spectral sequence has a single nonzero row, so it collapses at
$E_2$ and $\T^{p,m}(A)=\HH^p(A,{}_1A_{\nu^m})$. For $m<0$ the bimodule
${}_1A_\nu$ is invertible with inverse ${}_1A_{\nu^{-1}}$, so the same holds
with negative $m$.

\emph{(2) Gorenstein amplitude.} Let $d_A=\operatorname{inj.dim}{}_AA
=\operatorname{inj.dim}A_A$ (finite and equal, $A$ being Gorenstein). Then
$\pd(\D A)_A=\operatorname{inj.dim}{}_AA=d_A$: dualizing a projective
resolution of ${}_AA$ to an injective coresolution of $(\D A)_A$ and using
$\operatorname{inj.dim}=\pd$ of the $\kk$-dual. The derived tensor factor
$\omega\Lotimes_A(-)$ therefore raises the upper homological amplitude of a
bounded complex by at most $d_A$; applying it $m$ times to $\omega$ (which has
amplitude $[0,d_A]$) gives $T_q=0$ for $q>(m-1)d_A$, i.e.\ at most
$(m-1)d_A+1$ nonzero rows.

\emph{(3) Hereditary two-row degeneration.} If $\gl A=1$ then for $m=2$ the
complex $\omega^{\Lotimes2}$ has homology in degrees $0,1$ only, so it is a
two-term Postnikov system: there is an exact triangle
$T_1[1]\to\omega^{\Lotimes2}\to T_0\xrightarrow{\varepsilon}T_1[2]$ whose
connecting class $\varepsilon\in\Ext^2_{\Ae}(T_0,T_1)$ is the $k$-invariant of
the tower. Applying $\rHom_{\Ae}(A,-)$ and rotating gives the long exact
sequence displayed, in which the connecting homomorphism is composition
(Yoneda product) with $\varepsilon$; this is the standard identification of the
sole possibly-nonzero differential $d_2$ of a two-row spectral sequence with
the $k$-invariant of the corresponding two-stage Postnikov tower.
\end{proof}

\begin{remark}\label{rem:tau-edge}
The edge map
$\T^{p,m}(A)\to E_\infty^{p,0}\subseteq\HH^p(A,\,\D A^{\otimes_A m})$ compares the
derived twisted theory with its underived truncation. For $m=2$ the input row
\[
E_2^{\bullet,q}=\HH^\bullet\bigl(A,\ \Tor^A_q(\D A,\D A)\bigr)
\]
is built from the derived shadow of \cite{ArmTau}; the short exact sequences
$0\to\mathrm B_n\to\tau_n\Lambda\to\Tor_n^\Lambda(\D\Lambda,\D\Lambda)\to0$
of \cite[Theorem 3.2]{ArmTau} present the same tower at the level of
Happel's minimal model. The extension class governing $d_2$ in (3) is precisely
the class studied in the remarks of \cite[\S7.1]{ArmTau}.
\end{remark}

\section{The Serre reflection}\label{sec:refl}

The two theorems of this section share one mechanism: \emph{the Serre functor
of the enveloping algebra evaluated at the diagonal bimodule is the square of
the Serre bimodule},
\begin{equation}\label{eq:square}
S_{\Ae}(A)\;=\;\D(\Ae)\Lotimes_{\Ae}A\;\simeq\;\omega\Lotimes_A\omega,
\end{equation}
which is \cite[Theorem 3.2(4)]{ArmTau} (the collapse
$\D(\Ae)\otimes_{\Ae}P\cong\D A\otimes_AP\otimes_A\D A$ on projective
bimodules). This explains the reflection center $m=1=\tfrac{0+2}{2}$.

\begin{theorem}[Smooth proper reflection]\label{thm:reflsmooth}
Let $A$ be a smooth and proper dg algebra over $\kk$ (for an ordinary
algebra: $\gl A<\infty$ with $E$ separable). Then for all
$(p,m)\in\mathbb Z^2$,
\[
\T^{p,m}(A)\;\cong\;\D\,\T^{-p,\,2-m}(A).
\]
\end{theorem}

\begin{proof}
Since $A$ is smooth and proper with $E$ separable, so is $\Ae$, and
$A\in\per(\Ae)$; moreover $\omega=\D A$ is invertible in $\der(\Ae)$
\cite[Lemma 7.2]{ArmCox} (its inverse is $\rHom_{\Ae}(\D A,\Ae)$), whence
$\omega^{\Lotimes m}[p]\in\per(\Ae)$ for every $m\in\Z$, $p\in\Z$. A smooth
proper category has a Serre functor $S_{\Ae}$ on $\per(\Ae)$ and a natural
perfect pairing
\[
\Hom_{\der(\Ae)}(U,V)\ \cong\ \D\,\Hom_{\der(\Ae)}\bigl(V,\ S_{\Ae}(U)\bigr)
\qquad(U,V\in\per(\Ae))
\]
\cite{BondalKapranov,Shklyarov}. Take $U=A$, $V=\omega^{\Lotimes m}[p]$. Using
the key identity \eqref{eq:square}, $S_{\Ae}(A)\simeq\omega^{\Lotimes2}$, and
then twisting by the invertible object $\omega^{\Lotimes(-m)}$ (an
autoequivalence of $\per(\Ae)$, hence bijective on $\Hom$-groups),
\begin{align*}
\T^{p,m}(A)=\Hom(A,\omega^{\Lotimes m}[p])
&\cong\D\Hom\bigl(\omega^{\Lotimes m}[p],\ S_{\Ae}(A)\bigr)\\
&\cong\D\Hom\bigl(\omega^{\Lotimes m}[p],\ \omega^{\Lotimes2}\bigr) \\
&=\D\Hom\bigl(\omega^{\Lotimes m},\ \omega^{\Lotimes2}[-p]\bigr)\\
&\cong\D\Hom\bigl(A,\ \omega^{\Lotimes(2-m)}[-p]\bigr) \\
&=\D\,\T^{-p,\,2-m}(A). \qedhere
\end{align*}
\end{proof}

\begin{theorem}[Tate reflection]\label{thm:refltate}
Let $A$ be self-injective. Then for all $n,m\in\mathbb Z$,
\[
\sExt^{\,n}_{\Ae}\bigl(A,\ \omega^{\otimes m}\bigr)\;\cong\;
\D\,\sExt^{\,-n-1}_{\Ae}\bigl(A,\ \omega^{\otimes(2-m)}\bigr).
\]
In particular, the Tate extension of the Han column is self-dual, and for
$n\ge1$ these groups compute $\D\HH_{n}$ of the twisted theories.
\end{theorem}

\begin{proof}
Set $\Lambda:=\Ae$. Since $A$ is self-injective, $\Lambda$ is a self-injective
(indeed Frobenius) algebra, and its stable module category
$\underline{\mathrm{mod}}\,\Lambda$ is a triangulated category with suspension
$\Sigma=\Omega^{-1}$ possessing a Serre functor. We use the following two
inputs, quoted precisely.

\begin{quote}
\emph{(Nakayama functor.)} For a self-injective algebra $\Lambda$ the stable
Nakayama functor is $\nu_\Lambda=\D\Lambda\Lotimes_\Lambda(-)$, an exact
autoequivalence of $\underline{\mathrm{mod}}\,\Lambda$; on $\Lambda$ itself
$\nu_\Lambda(\Lambda)=\D\Lambda$. If $\Lambda$ is Frobenius with Nakayama
automorphism $\pi$ then, under the convention $\D\Lambda\cong{}_1\Lambda_\pi$
fixed below, $\nu_\Lambda(M)={}_1\Lambda_\pi\otimes_\Lambda M$ for every module
$M$.
\end{quote}

\begin{quote}
\emph{(Triangulated Serre duality \cite[Ch.~IV]{ARS}.)} The Serre functor of
the triangulated category $\underline{\mathrm{mod}}\,\Lambda$ is a
\emph{single} syzygy twist of the Nakayama functor,
$S=\Omega\,\nu_\Lambda\;(=\tau\Sigma$, where $\tau=\Omega^2\nu_\Lambda$ is the
Auslander--Reiten translate$)$, and there is a natural perfect pairing
\[
\sHom_\Lambda(X,Y)\ \cong\ \D\,\sHom_\Lambda\bigl(Y,\ SX\bigr)
=\D\,\sHom_\Lambda\bigl(Y,\ \Omega\,\nu_\Lambda X\bigr)
\qquad(X,Y\in\underline{\mathrm{mod}}\,\Lambda).
\]
\end{quote}

\emph{Conventions.} For an algebra automorphism $\phi$, write ${}_1A_\phi$ for
$A$ with left action untwisted and right action twisted by $\phi$:
$a\cdot x\cdot b=ax\phi(b)$; the paper's convention (\S\ref{sec:intro}) is
$\omega=\D A\cong{}_1A_\nu$, and likewise $\D\Lambda\cong{}_1\Lambda_\pi$ with
$\pi$ the Nakayama automorphism of $\Lambda$. Two rules, verified once here.
\begin{itemize}
\item[\textup{(T1)}] \emph{Twist composition:}
${}_1A_\phi\otimes_A{}_1A_\psi\cong{}_1A_{\phi\psi}$. Indeed
$x\otimes y\mapsto x\,\phi(y)$ is well defined on the balanced tensor
(both $x\phi(a)\otimes y$ and $x\otimes ay$ map to $x\phi(ay)$),
left-linear, and sends the right action $(x\otimes y)b=x\otimes y\psi(b)$ to
$x\phi(y)\phi\psi(b)$, i.e.\ right-twist by $\phi\psi$. (With $\phi=\psi=\nu$
this gives $\omega^{\otimes2}={}_1A_\nu\otimes_A{}_1A_\nu\cong{}_1A_{\nu^2}$.)
\item[\textup{(T2)}] \emph{Both-sided to right normal form:}
${}_\phi A_\psi\cong{}_1A_{\phi^{-1}\psi}$ via $x\mapsto\phi^{-1}(x)$
(here ${}_\phi A_\psi$ has action $a\cdot x\cdot b=\phi(a)x\psi(b)$; then
$\phi^{-1}(\phi(a)x\psi(b))=a\,\phi^{-1}(x)\,\phi^{-1}\psi(b)$).
\item[\textup{(T3)}] \emph{Restriction form of the twist:} for an automorphism
$\pi$ of $\Lambda$ and a module $M$, ${}_1\Lambda_\pi\otimes_\Lambda M\cong
\mathrm{res}_{\pi^{-1}}(M)$ via $\ell\otimes m\mapsto\pi^{-1}(\ell)m$. (Left
$\Lambda$-linearity: $\lambda\ell\otimes m\mapsto\pi^{-1}(\lambda)\pi^{-1}(\ell)m$,
the $\mathrm{res}_{\pi^{-1}}$-action; balancing over $\ell\pi(a)\otimes m=
\ell\otimes am$: both sides map to $\pi^{-1}(\ell)\,a\,m$.)
\end{itemize}

\emph{Step 1: $\nu_\Lambda(A)\cong\omega^{\otimes2}$, computed directly on the
self-injective locus.} By the Nakayama-functor input,
$\nu_\Lambda(A)={}_1\Lambda_\pi\otimes_\Lambda A$ with $\pi$ the Nakayama
automorphism of $\Lambda=A\otimes_\kk A^{\mathrm{op}}$. The Nakayama
automorphism of a tensor product of Frobenius algebras is the tensor of the
factors' Nakayama automorphisms; since $A$ has Nakayama automorphism $\nu$ and
$A^{\mathrm{op}}$ has $\nu^{-1}$, we get $\pi=\nu\otimes\nu^{-1}$, acting on
$\Lambda$ by $a\otimes b^{\mathrm{op}}\mapsto\nu(a)\otimes\nu^{-1}(b)^{\mathrm{op}}$.
By (T3), $\nu_\Lambda(A)\cong\mathrm{res}_{\pi^{-1}}(A)$ where
$\pi^{-1}=\nu^{-1}\otimes\nu$. On the regular bimodule ${}_1A_1$ (whose
$\Lambda$-action is $(a\otimes b^{\mathrm{op}})\cdot x=axb$), restricting along
$\pi^{-1}$ replaces this by
\[
(a\otimes b^{\mathrm{op}})\ast x
=\bigl(\nu^{-1}(a)\otimes\nu(b)^{\mathrm{op}}\bigr)\cdot x
=\nu^{-1}(a)\,x\,\nu(b),
\]
i.e.\ the bimodule ${}_{\nu^{-1}}A_{\nu}$. By (T2) with $\phi=\nu^{-1}$,
$\psi=\nu$,
\[
\nu_\Lambda(A)\cong{}_{\nu^{-1}}A_{\nu}\cong{}_1A_{(\nu^{-1})^{-1}\nu}
={}_1A_{\nu^2}=\omega^{\otimes2}.
\]
This is the self-injective, module-level instance of \cite[Theorem 3.2(4)]{ArmTau};
no smoothness hypothesis and no reference to \eqref{eq:square} is used: the
computation is entirely inside honest bimodules, valid precisely because $A$ is
self-injective so that $\omega\cong{}_1A_\nu$ is invertible.

\emph{Step 2: the Tate-degree bookkeeping.} Recall
$\sExt^{\,n}_\Lambda(U,V)=\sHom_\Lambda(U,\Sigma^nV)=\sHom_\Lambda(\Omega^nU,V)$.
Take $X=A$, and apply triangulated Serre duality to the pair
$(A,\ \Sigma^nY)$:
\[
\sExt^{\,n}_\Lambda(A,Y)=\sHom_\Lambda(A,\Sigma^nY)
\ \cong\ \D\,\sHom_\Lambda\bigl(\Sigma^nY,\ \Omega\,\nu_\Lambda A\bigr).
\]
Now $\sHom_\Lambda(\Sigma^nY,\Omega\nu_\Lambda A)
=\sHom_\Lambda(Y,\Sigma^{-n}\Omega\nu_\Lambda A)
=\sHom_\Lambda(Y,\Omega^{\,n+1}\nu_\Lambda A)
=\sExt^{\,-n-1}_\Lambda(Y,\nu_\Lambda A)$,
the last equality because $\Omega^{\,n+1}=\Sigma^{-(n+1)}$ so
$\sHom_\Lambda(Y,\Omega^{n+1}Z)=\sHom_\Lambda(Y,\Sigma^{-n-1}Z)
=\sExt^{\,-n-1}_\Lambda(Y,Z)$. Hence
\[
\sExt^{\,n}_\Lambda(A,Y)\ \cong\ \D\,\sExt^{\,-n-1}_\Lambda\bigl(Y,\ \nu_\Lambda A\bigr).
\]
Substituting $Y=\omega^{\otimes m}$ and $\nu_\Lambda A\cong\omega^{\otimes2}$
(Step 1),
\[
\sExt^{\,n}_\Lambda(A,\omega^{\otimes m})
\ \cong\ \D\,\sExt^{\,-n-1}_\Lambda\bigl(\omega^{\otimes m},\ \omega^{\otimes2}\bigr).
\]

\emph{Step 3: untwist.} The functor $-\otimes_A\omega^{\otimes(-m)}
=-\otimes_A{}_1A_{\nu^{-m}}$ is an exact self-equivalence of
$\mathrm{mod}\,\Lambda$ carrying projectives to projectives (invertible, inverse
$-\otimes_A{}_1A_{\nu^m}$), hence a triangulated autoequivalence of
$\underline{\mathrm{mod}}\,\Lambda$ and a bijection on $\sExt$; by (T1)
$\omega^{\otimes m}\otimes_A\omega^{\otimes(-m)}\cong A$ and
$\omega^{\otimes2}\otimes_A\omega^{\otimes(-m)}\cong\omega^{\otimes(2-m)}$, so
\[
\sExt^{\,-n-1}_\Lambda(\omega^{\otimes m},\omega^{\otimes2})
\cong\sExt^{\,-n-1}_\Lambda\bigl(A,\ \omega^{\otimes(2-m)}\bigr).
\]
Chaining Steps 2--3 gives
$\sExt^{\,n}_\Lambda(A,\omega^{\otimes m})\cong
\D\,\sExt^{\,-n-1}_\Lambda(A,\omega^{\otimes(2-m)})$. For $n\ge1$, Tate and
ordinary $\Ext$ agree, and $\Ext^n_\Lambda(A,\D A)\cong\D\HH_n$ by
Lemma~\ref{lem:dual}, giving the last clause. In the symmetric case
$\omega\cong A$ this recovers Bergh--Jorgensen's Tate--Hochschild self-duality
\cite{BerghJorgensen}, a consistency check on the $-n-1$ shift.
\end{proof}

\begin{remark}
For symmetric $A$ ($\omega\cong A$) Theorem \ref{thm:refltate} is the
Tate--Hochschild duality of Bergh--Jorgensen \cite{BerghJorgensen}. For
Gorenstein $A$ the same statement holds with $\sExt$ computed in the
singularity category of $\Ae$ in the sense of Buchweitz \cite{Buchweitz}. A
hypothetical counterexample to Han's conjecture has a bounded, Tate-self-dual
Han column: the vanishing of the positive tail forces the vanishing of the
negative Tate tail. For periodic algebras this rigidity is exploited in
Theorem~\ref{thm:periodic}.
\end{remark}

\section{Algebraic geometry}\label{sec:geom}

Throughout this section $\kk$ is algebraically closed of characteristic $0$,
$X$ is a smooth projective variety of dimension $d$ with a tilting object
$T\in\der^b(\operatorname{coh}X)$, and $A=\operatorname{End}_X(T)$, so that
$\per A\simeq\der^b(\operatorname{coh}X)$ and $\gl A<\infty$. We write
$\omega_X$ for the canonical bundle and $K_X$ for the canonical divisor.

\begin{theorem}[Dictionary]\label{thm:geom}
For all $(p,m)\in\mathbb Z^2$,
\[
\T^{p,m}(A)\;\cong\;\bigoplus_{q\ge0}
H^{\,p+md-q}\bigl(X,\ \Lambda^qT_X\otimes\omega_X^{\otimes m}\bigr).
\]
\end{theorem}

\begin{proof}
\emph{Step 1: transport to the geometric side.} By To\"en's derived Morita
theory \cite{Toen}, the tilting object $T$ induces an equivalence
$\Phi\colon\der(\Ae)\xrightarrow{\ \sim\ }\der_{\mathrm{qc}}(X\times X)$ of
monoidal categories, under which the diagonal bimodule $A$ corresponds to the
structure sheaf of the diagonal, $\Phi(A)=\mathcal O_\Delta$, and the derived
tensor product $\Lotimes_A$ corresponds to convolution (composition) of
Fourier--Mukai kernels. In particular, $\Hom_{\der(\Ae)}$ becomes
$\Hom_{\der(X\times X)}$.

\emph{Step 2: identify $\Phi(\omega_A^{\Lotimes m})$.} The equivalence
intertwines the Serre functors of $\per A$ and $\der^b(\operatorname{coh}X)$.
The Serre functor of $\der^b(\operatorname{coh}X)$ is $-\otimes\omega_X[d]$,
whose Fourier--Mukai kernel is $\Delta_*(\omega_X)[d]$; the Serre functor of
$\per A$ is $-\Lotimes_A\omega_A[?]$ with kernel $\omega_A$ up to shift. By
uniqueness of Serre functors these kernels correspond:
$\Phi(\omega_A)\cong\Delta_*(\omega_X)[d]$. Convolution of diagonal kernels
multiplies the line bundles, $\Delta_*L\circ\Delta_*L'\cong\Delta_*(L\otimes L')$,
so by induction
\[
\Phi\bigl(\omega_A^{\Lotimes m}\bigr)\cong\Delta_*\bigl(\omega_X^{\otimes m}\bigr)[md].
\]

\emph{Step 3: compute the $\Hom$.} Therefore
\[
\T^{p,m}(A)=\Hom_{\der(\Ae)}(A,\omega_A^{\Lotimes m}[p])
\cong\Ext^{\,p+md}_{X\times X}\bigl(\mathcal O_\Delta,\ \Delta_*\omega_X^{\otimes m}\bigr).
\]
By adjunction $\Delta_*\dashv L\Delta^*$ (and the projection formula),
$\Ext^\bullet_{X\times X}(\mathcal O_\Delta,\Delta_*\mathcal F)
\cong\Ext^\bullet_X(L\Delta^*\mathcal O_\Delta,\mathcal F)$, so
\[
\T^{p,m}(A)\cong\Ext^{\,p+md}_X\bigl(L\Delta^*\mathcal O_\Delta,\ \omega_X^{\otimes m}\bigr).
\]

\emph{Step 4: HKR.} The Hochschild--Kostant--Rosenberg isomorphism
\cite{Caldararu,Swan,Yekutieli} gives, for $X$ smooth over a field of
characteristic $0$, $L\Delta^*\mathcal O_\Delta\cong\bigoplus_{q\ge0}\Omega^q_X[q]$.
Substituting and using, at the total degree $p+md$,

\begin{align*}
    \Ext^{\,p+md}_X(\Omega^q_X[q],\omega_X^{\otimes m})
&=\Ext^{\,p+md-q}_X(\Omega^q_X,\omega_X^{\otimes m}) \\
&=H^{\,p+md-q}(X,(\Omega^q_X)^\vee\otimes\omega_X^{\otimes m}) \\
&=H^{\,p+md-q}(X,\Lambda^qT_X\otimes\omega_X^{\otimes m})
\end{align*}

the shift $[q]$ lowering the cohomological degree by $q$, yields
\[
\T^{p,m}(A)\cong\bigoplus_{q\ge0}H^{\,p+md-q}\bigl(X,\ \Lambda^qT_X\otimes\omega_X^{\otimes m}\bigr).\qedhere
\]
\end{proof}

\begin{corollary}[The Han column is Hodge cohomology]\label{cor:hodge}
$\T^{p,1}(A)\cong\bigoplus_{j}H^{\,p+j}(X,\Omega^j_X)$. Equivalently, via
$\T^{p,1}=\D\HH_p(A)$: dual Hochschild homology of a geometric algebra is the
Hodge cohomology of $X$, and the Serre reflection
(Theorem~\ref{thm:reflsmooth}) restricted to the Han column is classical
Serre duality $H^{p+j}(\Omega^j)\cong\D H^{d-p-j}(\Omega^{d-j})$.
\end{corollary}

\begin{proof}
Put $m=1$ in Theorem \ref{thm:geom} and use
$\Lambda^qT_X\otimes\omega_X\cong\Omega^{d-q}_X$, reindexing $j=d-q$.
\end{proof}

\begin{corollary}[Canonical ring]\label{cor:canring}
The assignment $H^0(X,\omega_X^{\otimes m})\ni s\mapsto
[s]\in\T^{-md,m}(A)$ embeds the canonical ring $R(X,K_X)$ as a bigraded
subalgebra of the plane. Dually the anticanonical ring embeds in the columns
$m\le0$.
\end{corollary}

\begin{proof}
The $(q,i)=(0,0)$ summand of Theorem \ref{thm:geom} in weight $m$ is
$H^0(X,\omega_X^{\otimes m})$ sitting in $\T^{-md,m}$; the plane product of
\S\ref{subsec:products} corresponds to composition of kernels
$\Delta_*\omega^m$, i.e.\ to multiplication of sections.
\end{proof}

\begin{theorem}[Column growth and Kodaira-type invariants]\label{thm:growth}
Let $t_m:=\sum_p\dim_\kk\T^{p,m}(A)<\infty$. Then:
\begin{enumerate}
\item $t_m=t_{2-m}$ for all $m$ (Theorem \ref{thm:reflsmooth});
\item if $\omega_X^{\otimes r}\cong\mathcal O_X$ for some $r\ge1$
(torsion canonical class; e.g.\ Calabi--Yau) then $t_{m+r}=t_m$ for all $m$:
the plane is periodic;
\item if $K_X$ or $-K_X$ is ample, then
$t_m=2^d\,\frac{|K_X^d|}{d!}\,|m|^d\,(1+o(1))$ as $|m|\to\infty$;
\item in general $t_m=O(|m|^d)$ and
$t_m\ge h^0(X,\omega_X^{\otimes m})$, so
$\limsup_m\frac{\log t_m}{\log m}\ge\kappa(X)$ whenever $\kappa(X)\ge0$.
\end{enumerate}
Consequently
$\kappa_\sigma(A):=\limsup_{m\to\infty}\log t_m/\log m$
(the \emph{Serre--Kodaira dimension} of $A$, a derived invariant by
Theorem~\ref{thm:derinv}) satisfies
$\kappa(X)\le\kappa_\sigma(A)\le d$, with $\kappa_\sigma(A)=d$ when $\pm K_X$
is ample and $\kappa_\sigma(A)=0$ when $K_X$ is torsion.
\end{theorem}

\begin{proof}
(1) is Theorem \ref{thm:reflsmooth}. (2) is immediate from Theorem
\ref{thm:geom}. (3) For $K_X$ ample and $m\to+\infty$, Serre vanishing kills
$H^{>0}(\Lambda^qT_X\otimes\omega_X^m)$, and asymptotic Riemann--Roch gives
$h^0(\Lambda^qT_X\otimes\omega_X^m)=\operatorname{rk}(\Lambda^qT_X)\,
\frac{K^d}{d!}m^d(1+o(1))$; summing over $q$ yields the constant
$\sum_q\binom dq=2^d$. The case $-K_X$ ample follows from (1). (4) The upper
bound is standard for twists by powers of a fixed line bundle; the lower
bound is the $(q,i)=(0,0)$ summand. Derived invariance of $t_m$ is
Theorem~\ref{thm:derinv}.
\end{proof}

\begin{remark}
\begin{itemize}
    \item (i) Via Corollary \ref{cor:canring} and Theorem \ref{thm:derinv}, the plane
recovers, on the level of finite-dimensional algebras, the invariance of
(anti)canonical rings under derived equivalence of smooth projective
varieties \cite{OrlovSurvey}.
    \item  (ii) The finer growth of $t_m$ for arbitrary
$X$ is governed by the asymptotic cohomology functions
$\widehat h^i(K_X)$ of de Fernex--K\"uronya--Lazarsfeld \cite{dFKL}; the
plane packages them into a derived invariant of the algebra.
    \item  (iii) For
singular algebras, the role of $t_m$-growth is taken by the cap support of
Section~\ref{sec:fg}; comparing the two invariants across the
smooth/singular divide looks like an interesting problem.
\end{itemize}

\end{remark}

\begin{example}[The projective line]\label{ex:P1}
Let $A=\operatorname{End}(\mathcal O\oplus\mathcal O(1))$ be the Kronecker
algebra, $X=\mathbb P^1$, $d=1$, $T_X=\mathcal O(2)$,
$\omega_X=\mathcal O(-2)$. Theorem \ref{thm:geom} gives two summands per
weight, $H^{p+m}(\mathcal O(-2m))$ and $H^{p+m-1}(\mathcal O(2-2m))$, whence
the full plane:
\[
\begin{array}{c|cc}
m & \text{nonzero }\T^{p,m} &\\
\hline
m\le0 & \T^{-m,m}=\kk^{\,1-2m}, & \T^{1-m,m}=\kk^{\,3-2m}\\
m=1 & \T^{0,1}=\kk^{2} & \\
m\ge2 & \T^{1-m,m}=\kk^{\,2m-1}, & \T^{2-m,m}=\kk^{\,2m-3}
\end{array}
\]
Thus $t_m=|4m-4|$ for $m\ne1$: linear growth, $\kappa_\sigma=1=\dim X$ with
the predicted constant $2^1\cdot\tfrac{|K^1|}{1!}=4$ (Theorem
\ref{thm:growth}(3), $-K$ ample); the reflection $t_m=t_{2-m}$ is visible.
The column $m=0$ is $(\HH^0,\HH^1)=(\kk,\kk^3)$:
$\HH^1(A)\cong H^0(T_{\mathbb P^1})=\mathfrak{sl}_2$, and the Han column
$m=1$ is the Hodge cohomology $\kk^2$ of $\mathbb P^1$, i.e.\
$\D\HH_0(A)=\D(A/[A,A])$. Section \ref{sec:comp} recovers the columns
$m=-1,0,1,2$ on the algebra side, through the
spectral sequence of Theorem \ref{thm:ss}.
\end{example}

\begin{example}[Calabi--Yau collapse]\label{ex:ell}
For a smooth proper dg algebra $\mathcal A$ with
$\per\mathcal A\simeq\der^b(E)$, $E$ an elliptic curve, Theorem
\ref{thm:geom} (in its dg form, with the same proof) gives
$\T^{p,m}=H^{p+m}(\mathcal O_E)\oplus H^{p+m-1}(\mathcal O_E)$: every column
equals $(\kk,\kk^2,\kk)$ along $p+m\in\{0,1,2\}$; the plane is constant in
$m$, in accordance with Theorem \ref{thm:growth}(2) and with the degeneration
of the Coxeter automorphism on Calabi--Yau categories
\cite[Example 7.8]{ArmCox}.
\end{example}

\begin{example}[The projective plane: the dictionary in dimension two]\label{ex:P2}
Let $X=\mathbb P^2$ and $A=\operatorname{End}(\mathcal O\oplus\mathcal O(1)\oplus\mathcal O(2))$
the Beilinson algebra, so $\per A\simeq\der^b(\operatorname{coh}\mathbb P^2)$,
$d=2$, $\omega_X=\mathcal O(-3)$, $T_X\cong\Omega^1_X(3)$. Theorem
\ref{thm:geom} gives three summands per weight,
\[
\T^{p,m}(A)\cong H^{p+2m}(\mathcal O(-3m))\ \oplus\ H^{p+2m-1}(\Omega^1_X(3-3m))
\ \oplus\ H^{p+2m-2}(\mathcal O(3-3m)),
\]
each evaluated by Bott's formula. The low-weight columns are
\[
\begin{array}{c|l}
m & \text{nonzero }\T^{p,m}\\
\hline
m=0 & \T^{0,0}=\kk,\ \ \T^{1,0}=\kk^{8},\ \ \T^{2,0}=\kk^{10}\\
m=1 & \T^{0,1}=\kk^{3}\\
m=2 & \T^{-2,2}=\kk^{10},\ \ \T^{-1,2}=\kk^{8},\ \ \T^{0,2}=\kk\\
m=-1 & \T^{2,-1}=\kk^{10},\ \ \T^{3,-1}=\kk^{35},\ \ \T^{4,-1}=\kk^{28}
\end{array}
\]
The column $m=0$ is polyvector cohomology: $\T^{1,0}=\HH^1(A)\cong H^0(T_{\mathbb P^2})
=\mathfrak{pgl}_3$ (dimension $8$) and $\T^{2,0}=\HH^2(A)\cong H^0(\mathcal O(3))$
(dimension $10$, the anticanonical deformations). The Han column $m=1$ is the
Hodge cohomology, concentrated in $\T^{0,1}=\kk^3$ carrying the three classes
$h^{0,0}=h^{1,1}=h^{2,2}=1$; one-dimensional in total weight, it is manifestly
its own Serre reflection. The reflection $t_m=t_{2-m}$ holds throughout
($t_0=t_2=19$, $t_{-1}=t_3=73$, $t_{-2}=t_4=163$); as $-K_{\mathbb P^2}$ is
ample, $t_m=2^2\tfrac{|K_X^2|}{2!}\,m^2(1+o(1))=18\,m^2(1+o(1))$ (Theorem
\ref{thm:growth}(3); e.g.\ $t_{40}=27379$ against $18\cdot40^2=28800$), so
$\kappa_\sigma(A)=2=\dim X$. The anticanonical ring
$\bigoplus_{m\le0}H^0(\mathcal O(-3m))$ embeds in the columns $m\le0$ (Corollary
\ref{cor:canring}), the canonical ring being trivial. This is the
two-dimensional companion of Example~\ref{ex:P1}: the plane records the full
Hodge and polyvector data of $X$ dimension by dimension, via Bott's formula.
\end{example}

\section{Finiteness, support, and Han's conjecture}\label{sec:fg}

Recall the Snashall--Solberg finiteness condition \cite{SnashallSolberg,EHSST}:
$A$ satisfies $\Fg$ if $\HH^\bullet(A)$ is Noetherian and
$\Ext^\bullet_A(E,E)$ is a finitely generated $\HH^\bullet(A)$-module.
Equivalently \cite[Prop.~2.3]{NWW}: \emph{$\HH^\bullet(A,B)$ is a Noetherian
$\HH^\bullet(A)$-module for every $B\in\mathrm{mod}\,\Ae$.} 

Let $R := \mathrm{HH}^{\mathrm{ev}}(A)= \bigoplus _i HH^{2i}(A)$, the even Hochschild cohomology ring, a graded-commutative $\kk$-algebra, and let

$$M := \mathrm{HH}^\bullet(A,\mathrm{D}A) \cong \mathrm{D}!\Big(\bigoplus_n \mathrm{HH}_n(A)\Big)$$

We define the cap support of $A$ as $$\mathcal V_{\cap}(A) := \operatorname{Supp}_{R} M = V\big(\operatorname{Ann}_R M\big) \subseteq \operatorname{Spec} R,$$

that is, the support variety of the Han column over the even cohomology ring: the closed subset of homogeneous prime ideals of $R = \mathrm{HH}^{\mathrm{ev}}(A)$ containing the annihilator of $M$. Its dimension is the Krull dimension

$$\dim \mathcal V_{\cap}(A) = \dim_{\mathrm{Kr}}\big(R/\operatorname{Ann}_R M\big).$$

This is only a reasonable geometric object under the Snashall–Solberg finiteness condition $\Fg$: then $R$ is Noetherian and $M$ is a finitely generated $R$-module (via Nguyen–Wang– Witherspoon), so the support is a genuine closed subvariety of finite dimension, and $\dim\mathcal V_\cap(A)$ equals the complexity of $\mathrm{HH}_\bullet(A)$, i.e. the polynomial growth rate: the Hilbert–Serre argument gives $\dim_\kk \mathrm{HH}_n(A) \sim n^{\dim\mathcal V\cap(A) - 1}$.


\subsection{The trichotomy}

\begin{theorem}\label{thm:fg}
Let $A$ satisfy $\Fg$. Then:
\begin{enumerate}
\item Every column $\T^{\bullet,m}(A)$, $m\in\mathbb Z$, is a finitely
generated graded module over the Noetherian graded-commutative ring
$R=\HH^{\mathrm{ev}}(A)$, acting without ambiguity by Theorem
\ref{thm:ladder}; on the Han column the action is the cap product on
$\D\HH_\bullet(A)$.
\item The Poincar\'e series $\sum_n\dim_\kk\HH_n(A)\,t^n$ is rational with
poles at roots of unity. Exactly one of the following holds:
\begin{itemize}
\item[(a)] $\dim\Vcap(A)=0$: $\HH_n(A)=0$ for $n\gg0$;
\item[(b)] $\dim\Vcap(A)=1$: $\dim\HH_n(A)$ is eventually periodic and
nonzero in infinitely many degrees;
\item[(c)] $\dim\Vcap(A)\ge2$: $\dim\HH_n(A)$ grows polynomially of degree
$\dim\Vcap(A)-1$.
\end{itemize}
\item If $\HH_\bullet(A)$ has a vanishing tail, every positive-degree class
of $\HH^{\mathrm{ev}}(A)$ acts nilpotently on $\HH_\bullet(A)$ by cap
products.
\end{enumerate}
\end{theorem}

\begin{proof}
\emph{(1) Finite generation.} We use the Nguyen--Wang--Witherspoon form of the
Snashall--Solberg condition: under $\Fg$, $\HH^\bullet(A,B)$ is a Noetherian
$\HH^\bullet(A)$-module for \emph{every} $B\in\mathrm{mod}\,\Ae$
\cite[Prop.~2.3]{NWW}. For a column with honest bimodule coefficient (e.g.\ the
Han column, $B=\D A$) this is immediate. For a general column
$\T^{\bullet,m}$, the coefficient $\omega^{\Lotimes m}$ is a bounded complex
with finite-dimensional homologies $T_q\in\mathrm{mod}\,\Ae$; run the spectral
sequence of Theorem~\ref{thm:ss}, whose $E_2^{p,q}=\HH^p(A,T_q)$ is a finite
direct sum (over $q$) of Noetherian $\HH^\bullet(A)$-modules, hence Noetherian.
Noetherianity passes to subquotients (each $E_r$, $r\ge2$, is a subquotient of
$E_2$) and, through the finite filtration of the abutment, to
$\T^{\bullet,m}$. That the $\HH^\bullet(A)$-action is unambiguous (left $=$
right) is Theorem~\ref{thm:ladder}; that on the Han column it is the cap
product of \cite{ArmCap} is the axiomatic uniqueness \cite[Theorem 3.1]{ArmCap}
(the cap action is the unique action satisfying (QI)--(QIII), and the Yoneda
action satisfies them).

\emph{(2) The trichotomy.} Write $R=\HH^{\ev}(A)$, a finitely generated
graded-commutative Noetherian $\kk$-algebra, and $M=\HH^\bullet(A,\D A)
=\D\bigoplus_n\HH_n(A)$, a finitely generated graded $R$-module by (1). By the
graded Hilbert--Serre theorem the Poincar\'e series $\sum_n\dim_\kk M_n\,t^n$ is
a rational function with denominator a product of factors $(1-t^{d_i})$, $d_i$
the degrees of homogeneous generators of $R/\operatorname{Ann}_R(M)$; its poles
lie at roots of unity, and the growth rate of $\dim_\kk M_n$ is $c-1$ where
$c=\dim_{\mathrm{Kr}}\bigl(R/\operatorname{Ann}_R M\bigr)=\dim\Vcap(A)$ is the
Krull dimension of the support. This gives the three cases according as $c=0$,
$c=1$, $c\ge2$. It remains to justify, in case $c=1$, that the dimension
sequence is eventually periodic and \emph{nonzero in infinitely many degrees}
(no full period of zeros). Choose a Noether normalization: a homogeneous
element $\zeta\in R$ of degree $e>0$ with
$\kk[\zeta]\hookrightarrow R/\operatorname{Ann}_R M$ finite, so that $M$ is a
finitely generated $\kk[\zeta]$-module of Krull dimension $1$. Over the graded
PID $\kk[\zeta]$, the structure theorem gives
\[
M\ \cong\ \Bigl(\bigoplus_{i=1}^{a}\kk[\zeta](-d_i)\Bigr)\ \oplus\
(\text{finite-dimensional torsion}),\qquad a\ge1
\]
(else $\dim_\kk M<\infty$, i.e.\ $c=0$), where the free generators sit in
degrees $d_1,\dots,d_a$, \emph{shifted}, not all in degree $0$. The shifted
free summand $\kk[\zeta](-d_i)$ contributes $1$ to $\dim_\kk M_n$ in every
degree $n\ge d_i$ with $n\equiv d_i\pmod e$; thus the free part is
$e$-periodic (once $n>\max_i d_i$) and is supported on the residue classes
$\{d_i\bmod e\}$, not merely on $0\bmod e$. Adding the eventually-zero torsion,
$\dim_\kk M_n$ is eventually periodic of period (dividing) $e$, and \emph{some}
residue class mod $e$ carries the nonzero value $\#\{i:d_i\equiv n\}\ge1$.
Hence a \emph{full} period of consecutive zeros is impossible: it would
force $a=0$, i.e.\ $\dim_\kk M<\infty$ and $c=0$. This is the required
statement (nonzero in infinitely many degrees), independent of \emph{which}
residue classes carry the mass. Dualizing, the same holds for
$\dim_\kk\HH_n(A)$.

\emph{(3) Nilpotence in the vanishing-tail case.} A vanishing homology tail
means $M=\HH^\bullet(A,\D A)$ is finite-dimensional, i.e.\ $\dim\Vcap(A)=0$,
i.e.\ $R/\operatorname{Ann}_R M$ is Artinian, i.e.\ every positive-degree
element of $R$ is nilpotent on $M$. Under the cap action this is the asserted
nilpotence of every positive even class on $\HH_\bullet(A)$.
\end{proof}

The directed asymmetry certificate is logically a \emph{contrapositive}: it
concludes $\neg\Fg$ and therefore cannot sit under the standing $\Fg$
hypothesis of Theorem~\ref{thm:fg}. We state it separately; here and below,
$\cx(M)$ denotes \textit{the rate of growth} of a finitely generated graded module
(its Krull dimension by Hilbert--Serre).

\begin{corollary}[Directed $\neg\Fg$ certificate]\label{cor:asymcert}
Let $A$ be any finite-dimensional algebra. If cohomology has a vanishing tail
while homology does not (i.e.\ $\cx\HH_\bullet(A)>\cx\HH^\bullet(A)$, the
quantum-plane phenomenon of \cite{BerghErdmann,BGMS}), then $A$ does not
satisfy $\Fg$. Only this direction holds: homology outgrowing cohomology
certifies $\neg\Fg$, whereas the reverse asymmetry (cohomology outgrowing
homology) is $\Fg$-compatible.
\end{corollary}

\begin{proof}
If $A$ satisfied $\Fg$, then by Theorem~\ref{thm:fg}(1) both $\HH^\bullet(A)$
(supporting itself) and $\HH^\bullet(A,\D A)$ would be finitely generated over
$\HH^{\ev}(A)$, with $\cx\HH^\bullet(A)=\dim V(A)$ and
$\cx\HH_\bullet(A)=\dim\Vcap(A)$; and $\Vcap(A)=\Supp\HH^\bullet(A,\D A)
\subseteq\Supp\HH^\bullet(A)=V(A)$ (the cap module is supported on the
cohomology variety, $\HH^{\ev}$ acting through its own quotient), forcing
$\cx\HH_\bullet\le\cx\HH^\bullet$, contrary to the hypothesis. Hence $\neg\Fg$.
\end{proof}

\begin{remark}[The trichotomy does \emph{not} exclude vanishing tails]
\label{rem:notexcluded}
No support or finite-generation theory for $\HH_\bullet(A)$ as an
$\HH^\bullet(A)$-module appears in the literature; the Snashall--Solberg
theory concerns the action on $\Ext$-algebras of one-sided modules. Theorem
\ref{thm:fg} initiates it and shows that on the entire $\Fg$ locus the
dimension sequence $\dim\HH_n$ is eventually quasi-polynomial, falling into
exactly one of the three regimes. \emph{We stress that case (a), a vanishing
tail, is allowed by the trichotomy}: Theorem~\ref{thm:fg} does \emph{not}
prove that vanishing tails require $\neg\Fg$, and any earlier suggestion to the
contrary is an overstatement. What is proven is:
\begin{itemize}
    \item (i) part (a) is equivalent to $\gl A<\infty$ \emph{precisely if} Han holds;
    \item (ii) the \emph{one-directional} certificate of part (4);
    \item (iii) for self-injective algebras, the reduction of the next subsection.
\end{itemize}
On the symmetric locus a vanishing tail \emph{is}
excluded (Theorem~\ref{thm:periodic}(3)); on the self-injective locus it is equivalent to a
failure of $\ASnu$; the general $\Fg$ case remains open. Part
(a)$\Rightarrow$smoothness is exactly Han's conjecture on the $\Fg$ locus.
\end{remark}

\subsection{Self-injective algebras: the $\nu$-twisted Avrunin--Scott
reduction}

Let $\kk$ be perfect, $A$ self-injective with $\Fg$, $\Lambda=\Ae$,
$H=\HH^{\mathrm{ev}}(\Lambda)$; for $M,N\in\mathrm{mod}\,\Lambda$ let $V(M)$
denote the support variety of \cite{EHSST} and
$V(M,N)=\Supp_H\Ext^\bullet_\Lambda(M,N)$.

\begin{lemma}\label{lem:fgenv}
$\Fg$ for $A$ implies $\Fg$ for $\Ae$.
\end{lemma}

\begin{proof}
We verify the two clauses of $\Fg$ for $\Ae$: Noetherianity of
$\HH^\bullet(\Ae)$ and finite generation of $\Ext^\bullet_{\Ae}(E^{\mathrm e},E^{\mathrm e})$
over it.

\emph{Noetherianity.} By Le--Zhou \cite{LeZhou}, for finite-dimensional
$A$ there is an isomorphism of graded algebras
$\HH^\bullet(\Ae)\cong\HH^\bullet(A)\otimes_\kk\HH^\bullet(A^{\mathrm{op}})$
(one tensor factor being finite-dimensional is automatic in each degree, so the
K\"unneth map is an isomorphism). Moreover $\HH^\bullet(A^{\mathrm{op}})\cong
\HH^\bullet(A)$ (the opposite algebra has the opposite bimodule structure,
under which the bar complex is isomorphic). We record for later use that
\emph{$\Fg$ for $A$ implies $\Fg$ for $A^{\mathrm{op}}$}: the same
bar-complex isomorphism gives $\HH^\bullet(A^{\mathrm{op}})\cong\HH^\bullet(A)$
as graded algebras, and $\Ext^\bullet_{A^{\mathrm{op}}}(E^{\mathrm{op}},E^{\mathrm{op}})
\cong\Ext^\bullet_A(E,E)^{\mathrm{op}}$ (the opposite of the Yoneda algebra),
which is finitely generated over $\HH^\bullet(A^{\mathrm{op}})$ exactly when
$\Ext^\bullet_A(E,E)$ is over $\HH^\bullet(A)$; hence both clauses of $\Fg$
transfer from $A$ to $A^{\mathrm{op}}$. Thus $\HH^\bullet(\Ae)$ is a tensor
product of two Noetherian graded-commutative finitely-generated $\kk$-algebras.
A tensor product $R\otimes_\kk S$ of two such is again a finitely generated
$\kk$-algebra (generated by $R\otimes1$ and $1\otimes S$) and hence Noetherian
by the graded Hilbert basis theorem.

\emph{Separability of $E^{\mathrm e}$ and K\"unneth for $\Ext$.} Since $\kk$ is
perfect, $E=A/r$ is separable, so $E\otimes_\kk E^{\mathrm{op}}$ is again
semisimple; concretely $\rad(\Ae)=r\otimes A^{\mathrm{op}}+A\otimes r^{\mathrm{op}}$
and $E^{\mathrm e}:=\Ae/\rad(\Ae)\cong E\otimes_\kk E^{\mathrm{op}}$. Choose
projective resolutions $P_\bullet\to E$ over $A$ and $Q_\bullet\to E^{\mathrm{op}}$
over $A^{\mathrm{op}}$; then $P_\bullet\otimes_\kk Q_\bullet\to E^{\mathrm e}$ is
a projective resolution over $\Ae=A\otimes_\kk A^{\mathrm{op}}$ (the tensor
product of projectives is projective, and the K\"unneth theorem over the field
$\kk$ shows it is a resolution). Applying $\Hom$ and K\"unneth once more,
\[
\Ext^\bullet_{\Ae}(E^{\mathrm e},E^{\mathrm e})\cong
\Ext^\bullet_A(E,E)\otimes_\kk\Ext^\bullet_{A^{\mathrm{op}}}(E^{\mathrm{op}},E^{\mathrm{op}}).
\]
This is a tensor product of finitely generated modules (by $\Fg$ for $A$ and
for $A^{\mathrm{op}}$) over the tensor product $\HH^\bullet(A)\otimes
\HH^\bullet(A^{\mathrm{op}})=\HH^\bullet(\Ae)$ of the acting rings, hence
finitely generated. Both clauses of $\Fg$ hold for $\Ae$.
\end{proof}

\begin{lemma}\label{lem:vda}
$V_{\Ae}(\D A)=V_{\Ae}(A)$.
\end{lemma}

\begin{proof}
We use the following result.

\begin{theorem}[{Su\'arez-\'Alvarez \cite{SuarezAlvarez}; see also
\cite[Cor.~6.3]{ArmCox}}]\label{thm:SuA}
Let $A$ be a Frobenius $\kk$-algebra with Nakayama automorphism $\nu$. Then
$\nu$ acts trivially on Hochschild cohomology: the graded-algebra automorphism
$\nu^*$ induced by $\nu$ on $\HH^\bullet(A)=\Ext^\bullet_{\Ae}(A,A)$ is the
identity.
\end{theorem}

\emph{Setup of the twist.} The Frobenius identification $\D A\cong{}_1A_\nu$
exhibits $\D A$ as the restriction $\operatorname{res}_\alpha(A)$ of the regular
bimodule along the algebra automorphism $\alpha=1\otimes\nu$ of $\Ae$ (here
$1\otimes\nu$ acts as the identity on the left tensor factor and as $\nu$ on
$A^{\mathrm{op}}$; that ${}_1A_\nu=\operatorname{res}_{1\otimes\nu}A$ is the
definition of the twisted bimodule). Restriction along the algebra
automorphism $\alpha$ is an exact self-equivalence $\operatorname{res}_\alpha$
of $\mathrm{mod}\,\Ae$ that preserves projectives (it permutes the
indecomposable projective $\Ae$-modules), hence induces isomorphisms on all
$\Ext$-groups.

\emph{Semilinearity and its triviality.} The induced isomorphism
\[
\Ext^\bullet_{\Ae}(M,N)\cong\Ext^\bullet_{\Ae}(\operatorname{res}_\alpha M,
\operatorname{res}_\alpha N)
\]
is semilinear over the automorphism
$\alpha^*$ of the acting ring $H=\HH^{\ev}(\Ae)$. Under the K\"unneth
identification $H\subseteq\HH^\bullet(A)\otimes\HH^\bullet(A^{\mathrm{op}})$ of
Lemma~\ref{lem:fgenv}, the automorphism $\alpha=1\otimes\nu$ induces
$\alpha^*=\id\otimes\nu^*$. Now $A^{\mathrm{op}}$ is Frobenius with Nakayama
automorphism $\nu^{-1}$, so by Theorem~\ref{thm:SuA} $\nu^*=\id$ on
$\HH^\bullet(A^{\mathrm{op}})$. Hence $\alpha^*=\id\otimes\id=\id$ on $H$: the
support-preserving self-equivalence $\operatorname{res}_\alpha$ is
$H$-\emph{linear}, not merely semilinear.

\emph{Conclusion.} For any $M$, 
\begin{align*}
    V_{\Ae}(\operatorname{res}_\alpha M)
& =\Supp_H\Ext^\bullet_{\Ae}(\operatorname{res}_\alpha M,\operatorname{res}_\alpha M) \\
&=\Supp_H\Ext^\bullet_{\Ae}(M,M)\\ &=V_{\Ae}(M), 
\end{align*}
the middle equality because the
isomorphism is $H$-linear. Taking $M=A$ gives
$V_{\Ae}(\D A)=V_{\Ae}(\operatorname{res}_\alpha A)=V_{\Ae}(A)$.
\end{proof}

\begin{definition}
$A$ satisfies $\mathsf{(AS_\nu)}$ if
$\Ext^n_{\Ae}(A,\D A)=0$ for $n\gg0$ implies that
$V_{\Ae}(A)\cap V_{\Ae}(\D A)$ is trivial.
\end{definition}

\begin{theorem}\label{thm:asnu}
Let $\kk$ be perfect and $A$ self-injective satisfying $\Fg$.
\begin{enumerate}
\item If $A$ satisfies $\mathsf{(AS_\nu)}$, then Han's conjecture holds for
$A$; in fact a vanishing tail of $\HH_\bullet(A)$ forces $A$ semisimple.
\item $\mathsf{(AS_\nu)}$ holds when $A$ is symmetric, and when the supports
over $\Ae$ satisfy the Avrunin--Scott realization property
$V(M,N)=V(M)\cap V(N)$ (e.g.\ $A=\kk G$ \cite{AvruninScott}); in these cases
Han's conjecture holds unconditionally.
\item Unconditionally, a vanishing tail of $\HH_\bullet(A)$ forces
$V(A,\D A)$ trivial; by Lemma \ref{lem:vda} the entire distance to Han's
conjecture on this locus is the single inclusion
$$V(A)\cap V(\D A)\subseteq V(A,\D A)$$ for the one pair $(A,\D A)$.
\end{enumerate}
\end{theorem}

\begin{proof}
We use one further result.

\begin{theorem}[{Erdmann--Holloway--Snashall--Solberg--Taillefer
\cite[Thm.~1.5]{EHSST}; Solberg \cite[Thm.~5.9]{SolbergSurvey}}]\label{thm:EHSST}
Let $\Lambda$ be a self-injective algebra satisfying $\Fg$ and
$M\in\mathrm{mod}\,\Lambda$. Then the support variety $V_\Lambda(M)$ is trivial
(a point) if and only if $\pd_\Lambda M<\infty$; equivalently, if and only if
$M$ is projective.
\end{theorem}

\emph{(1)} A vanishing homology tail means, via Lemma~\ref{lem:dual},
$\Ext^n_{\Ae}(A,\D A)\cong\D\HH_n(A)=0$ for $n\gg0$; this is exactly the
hypothesis of $\ASnu$. Hence $V_{\Ae}(A)\cap V_{\Ae}(\D A)$ is trivial. By
Lemma~\ref{lem:vda}, $V_{\Ae}(\D A)=V_{\Ae}(A)$, so
$V_{\Ae}(A)=V_{\Ae}(A)\cap V_{\Ae}(\D A)$ is trivial. Now $\Ae$ is
self-injective (a tensor product of self-injective algebras is self-injective)
and satisfies $\Fg$ (Lemma~\ref{lem:fgenv}); by Theorem~\ref{thm:EHSST} applied to
$M=A$ over $\Lambda=\Ae$, trivial support forces $\pd_{\Ae}A<\infty$. By
Happel's identity $\gl A=\pd_{\Ae}A<\infty$; and a self-injective algebra of
finite global dimension is semisimple (a self-injective algebra with
$\gl<\infty$ has all simples of finite projective and injective dimension,
forcing $\gl=0$). Conversely a semisimple algebra has $\HH_{\ge1}=0$, so
indeed vanishing tail $\iff$ semisimple here.

\emph{(2)} \emph{Symmetric case.} If $A$ is symmetric then $\D A\cong A$ as
bimodules, so $\Ext^\bullet_{\Ae}(A,\D A)=\Ext^\bullet_{\Ae}(A,A)=\HH^\bullet(A)$
and $V(A,\D A)=V(A,A)=V(A)$: the pair variety \emph{equals} the intersection,
so $\ASnu$ holds trivially. Concretely a vanishing tail makes
$\HH^\bullet(A)=\Ext_{\Ae}(A,A)$ finite-dimensional, which under $\Fg$ forces
$\Ext^\bullet_A(E,E)$ finite-dimensional, i.e.\ $\gl A<\infty$, directly.
\emph{Group-algebra case.} For $A=\kk G$ the Avrunin--Scott realization theorem
\cite{AvruninScott} gives the tensor-product property $V(M,N)=V(M)\cap V(N)$
for all $\Ae$-modules, so $\ASnu$ holds, recovering Han for group algebras by
support theory.

\emph{(3)} is Theorem~\ref{thm:fg}(3): a vanishing tail makes
$\Ext^\bullet_{\Ae}(A,\D A)$ finite-dimensional, i.e.\ $V(A,\D A)$ trivial;
Lemma~\ref{lem:vda} then leaves the single inclusion
$V(A)\cap V(\D A)\subseteq V(A,\D A)$ as the whole distance to Han's
conjecture on this locus.
\end{proof}

\begin{remark}
The general two-variable converse ``$\Ext$-tail vanishing $\Rightarrow$
supports intersect trivially'' is not available for arbitrary self-injective
algebras: it is entangled with the failure of the tensor product property
for Hochschild supports \cite{BSS}. The point of Theorem \ref{thm:asnu} is
that Han's conjecture needs it only for the single pair $(A,\D A)$, which by
Theorem \ref{thm:ladder} is diagonally symmetric and by Lemma \ref{lem:vda}
has equal supports, and whose Tate column is self-dual (Theorem
\ref{thm:refltate}); the known obstructions to such converses are thereby
excluded.
\end{remark}

\subsection{Periodic algebras and the stable twisted center}

Recall that $A$ is \emph{periodic} of period $\pi$ if
$\Omega^\pi_{\Ae}(A)\cong A$; periodic algebras are self-injective
\cite{ErdmannSkowronski}, and the class includes preprojective algebras of
Dynkin type, Brauer tree algebras, blocks of quaternion type and weighted
surface algebras.

\begin{lemma}\label{lem:seam}
If $A$ is periodic of period $\pi$, then $\HH_{n+\pi}(A)\cong\HH_n(A)$ for
all $n\ge1$. Consequently a vanishing tail forces $\HH_n(A)=0$ for all
$n\ge1$.
\end{lemma}

\begin{proof}
Let $P_\bullet\to A$ be the minimal projective $\Ae$-resolution. Periodicity
$\Omega^\pi_{\Ae}(A)\cong A$ means that the syzygy of order $\pi$ is again $A$,
so by minimality the resolution is periodic: there is an isomorphism of
complexes $P_{\bullet+\pi}\cong P_\bullet$ in all degrees $\ge1$, compatible
with the differentials (the truncation $P_{\ge1}$ is a resolution of
$\Omega A$, and the $\pi$-fold shift matches it to a resolution of
$\Omega^{1+\pi}A\cong\Omega A$). Applying the right-exact functor
$A\otimes_{\Ae}(-)$, the resulting chain complex computing $\HH_\bullet(A)$ is
$\pi$-periodic in homological degrees $\ge1$; its homology therefore satisfies
$\HH_{n+\pi}(A)\cong\HH_n(A)$ for all $n\ge1$. The isomorphism can fail at the
seam $n=0$ because $P_0=\Ae$ contributes $A\otimes_{\Ae}\Ae=A$ there, breaking
periodicity in degree $0$ only. Consequently if $\HH_n(A)=0$ for $n\gg0$, the
$\pi$-periodicity pushes the vanishing down to \emph{all} $n\ge1$.
\end{proof}

\begin{proposition}\label{prop:stablecenter}
Let $A$ be Frobenius with Nakayama automorphism $\nu$. Then
\[
\sHom_{\Ae}(A,\D A)\;\cong\;Z_\nu(A)/H_\nu(A),
\]
where $Z_\nu(A)=\{z\in A:\ z\,\nu(a)=a\,z\ \ \forall a\}\cong\D(A/[A,A])$ is
the $\nu$-twisted center and
\[
H_\nu(A)=\Bigl\{\textstyle\sum_ix_iy_i\ :\ \sum_ix_i\otimes y_i\in A\otimes A,\
\ \sum_i ax_i\otimes y_i=\sum_ix_i\otimes y_i\nu(a)\ \forall a\Bigr\}
\]
is the \emph{$\nu$-twisted Higman ideal}. We call this quotient the
\emph{stable $\nu$-twisted center} of $A$.
\end{proposition}

\begin{proof}
\emph{Identifying $\Hom$.} A bimodule map $f\colon A\to{}_1A_\nu$ is determined
by $z:=f(1)$, and the bimodule condition $f(a\cdot1\cdot b)=a\cdot f(1)\cdot b$
in ${}_1A_\nu$ (where $a\cdot z\cdot b=az\nu(b)$) forces, taking $b=1$ and
$a=1$ respectively, $f(a)=az=z\nu(a)$; conversely any $z$ with $az=z\nu(a)$ for
all $a$ defines such an $f$. Hence $\Hom_{\Ae}(A,{}_1A_\nu)\cong Z_\nu(A)
=\{z:az=z\nu(a)\ \forall a\}$ via $f\mapsto f(1)$. Fixing a Frobenius form
$\lambda\colon A\to\kk$ with associated nondegenerate pairing
$\langle a,b\rangle=\lambda(ab)$, the map $z\mapsto\langle z,-\rangle$
identifies $Z_\nu(A)$ with the annihilator of $[A,A]$ under $\lambda$, i.e.\
with $\D(A/[A,A])$ (the defining relation $az=z\nu(a)$ dualizes exactly to
vanishing on commutators via the Nakayama twist of $\lambda$).

\emph{Identifying the projectively-trivial maps.} A map $f\colon A\to M$ of
$\Ae$-modules factors through a projective bimodule iff it lifts along any
fixed epimorphism $\pi\colon P\twoheadrightarrow M$ from a projective $P$ (if
$f=\pi g$ then $f$ factors through the projective $P$; conversely if
$f=\beta\gamma$ with $P'$ projective, lift $\beta$ along $\pi$ using
projectivity of $P'$). Take $M={}_1A_\nu$ and the canonical projective cover
\[
\mu_\nu\colon(A\otimes A)_\nu\twoheadrightarrow{}_1A_\nu,\qquad
x\otimes y\longmapsto xy,
\]
the $\nu$-twist of the multiplication map (here $(A\otimes A)_\nu$ denotes the
free bimodule $A\otimes A$ with right action twisted by $\nu$, so that $\mu_\nu$
is a bimodule map onto ${}_1A_\nu$). A bimodule map $A\to(A\otimes A)_\nu$ is
given by $1\mapsto\sum_i x_i\otimes y_i$ subject to the bimodule constraint
$\sum_i ax_i\otimes y_i=\sum_i x_i\otimes y_i\,\nu(a)$ for all $a$, which is the
displayed twisted Casimir condition. Composing with $\mu_\nu$ sends such a
tensor to $\sum_i x_iy_i\in Z_\nu(A)$. Therefore the subspace of $Z_\nu(A)$
consisting of maps that factor projectively is exactly
$H_\nu(A)=\{\sum_i x_iy_i:\sum_i x_i\otimes y_i\ \text{a twisted Casimir}\}$,
and $\sHom_{\Ae}(A,\D A)=\Hom_{\Ae}(A,{}_1A_\nu)/(\text{proj.})=Z_\nu(A)/H_\nu(A)$.
\end{proof}

\begin{theorem}\label{thm:periodic}
Let $A$ be periodic of period $\pi$, not semisimple.
\begin{enumerate}
\item For all $n\ge1$,
$\D\HH_n(A)\cong\sExt^{\,j}_{\Ae}(A,\D A)$ where $j\equiv n\pmod\pi$; in
particular the degree-zero Tate group is the stable $\nu$-twisted center:
$\sExt^{\,0}_{\Ae}(A,\D A)=Z_\nu(A)/H_\nu(A)$.
\item $\HH_\bullet(A)$ has a vanishing tail if and only if
$\sExt^{\,j}_{\Ae}(A,\D A)=0$ for \emph{all} $j=0,\dots,\pi-1$, i.e.\ iff
$A\perp\D A$ in $\underline{\mathrm{mod}}\,\Ae$. In particular
$Z_\nu(A)\neq H_\nu(A)$ implies Han's conjecture for $A$, with
$\HH_{k\pi}(A)\neq0$ for all $k\ge1$.
\item If $A$ is symmetric, then
$\sHom_{\Ae}(A,\D A)\cong\sEnd_{\Ae}(A)\ni[\id_A]\neq0$, so Han's conjecture
holds for symmetric periodic algebras in every characteristic.
\item If $A$ is local and $\operatorname{char}\kk=0$, Han's conjecture holds
for $A$.
\end{enumerate}
\end{theorem}

\begin{proof}
\emph{(1)} Since $A$ is periodic it is self-injective, hence so is $\Ae$; thus
for $n\ge1$ ordinary and Tate cohomology agree,
$\D\HH_n(A)\cong\Ext^n_{\Ae}(A,\D A)\cong\sExt^{\,n}_{\Ae}(A,\D A)$
(Lemma~\ref{lem:dual}). Periodicity $\Omega^\pi_{\Ae}(A)\cong A$ gives, on Tate
cohomology, $\sExt^{\,n}_{\Ae}(A,\D A)\cong\sExt^{\,n-\pi}_{\Ae}(\Omega^\pi A,\D A)
\cong\sExt^{\,n+\pi}_{\Ae}(A,\D A)$ for \emph{all} $n\in\Z$ (Tate cohomology
turns syzygies into degree shifts). Hence
$\D\HH_n(A)\cong\sExt^{\,j}_{\Ae}(A,\D A)$ for $j\equiv n\pmod\pi$; the
degree-zero group $\sExt^{\,0}_{\Ae}(A,\D A)=\sHom_{\Ae}(A,\D A)$ is the stable
$\nu$-twisted center by Proposition~\ref{prop:stablecenter}.

\emph{(2)} If $\HH_\bullet(A)$ has a vanishing tail, Lemma~\ref{lem:seam} forces
$\HH_n(A)=0$ for all $n\ge1$, so by (1) $\sExt^{\,j}_{\Ae}(A,\D A)=0$ for all
$j=1,\dots,\pi-1$, and applying (1) at any $n\equiv0$, $n=k\pi\ge\pi\ge1$, also
$\sExt^{\,0}=\D\HH_{k\pi}(A)=0$; thus all $\pi$ Tate groups vanish, i.e.\
$A\perp\D A$ in $\underline{\mathrm{mod}}\,\Ae$. Conversely if all
$\sExt^{\,j}=0$ then $\D\HH_n=0$ for all $n\ge1$ by (1). The last clause: if
$Z_\nu\ne H_\nu$ then $\sExt^{\,0}\ne0$, so by (1) $\D\HH_{k\pi}(A)\cong
\sExt^{\,0}\ne0$ for every $k\ge1$ with $k\pi\ge1$, giving Han.

\emph{(3)} For symmetric $A$ we have $\D A\cong A$ as bimodules, so
$\sHom_{\Ae}(A,\D A)\cong\sEnd_{\Ae}(A)$, which contains the class $[\id_A]$.
If $[\id_A]=0$ then $\id_A$ factors through a projective bimodule, making $A$ a
direct summand of a projective $\Ae$-module, i.e.\ $A$ projective over $\Ae$,
i.e.\ $A$ separable, i.e.\ semisimple, contradicting the hypothesis. Hence
$[\id_A]\ne0$, $\sExt^{\,0}\ne0$, and (2) gives $\HH_{k\pi}(A)\ne0$ for all
$k\ge1$: Han holds, in every characteristic.

\emph{(4)} Suppose $A$ local, $\mathrm{char}\,\kk=0$, with a vanishing tail. By
Lemma~\ref{lem:seam}, $\HH_n(A)=0$ for all $n\ge1$. The $\chi$-law
\cite[Theorem 5.2]{ArmPC} (char $0$, $E$ split; here $E=\kk$, $s=1$) gives
$\sum_{n\ge0}(-1)^n\dim\HH_n(A)=s=1$, so with $\HH_{\ge1}=0$ this reads
$\dim_\kk\HH_0(A)=\dim_\kk A/[A,A]=1$, i.e.\ $r=[A,A]$. This is impossible by
Lemma~\ref{lem:local} below unless $A=\kk$ (semisimple). Hence no vanishing
tail, and Han holds.
\end{proof}


\begin{lemma}\label{lem:local}
For any local finite-dimensional algebra $A$ one has $[A,A]\subseteq r^2$.
Hence $r=[A,A]$ occurs only for $A=\kk$, and
$\dim_\kk A/[A,A]\ \ge\ 1+\dim_\kk r/r^2\ \ge\ 2$ whenever $A\neq\kk$.
\end{lemma}

\begin{proof}
Write $a=\lambda+a'$, $b=\mu+b'$ with $a',b'\in r$; then
$[a,b]=[a',b']\in r^2$. If $r=[A,A]\subseteq r^2$ then $r=r^2=0$ by
Nakayama's lemma. Since $[A,A]\subseteq r^2$, the algebra $A/[A,A]$ surjects onto $A/r^2$, whence the dimension bound.
\end{proof}

\begin{corollary}\label{cor:lenzing}
In characteristic $0$, no nontrivial local algebra has $\HH_i(A)=0$ for all
$i\ge1$. In the terminology of \cite[Corollary 5.5]{ArmPC}: the ``Lenzing
boundary'' locus $r=[A,A]$ is empty for local algebras, and every
hypothetical local counterexample to Han's conjecture must carry nonzero
homology in low positive degrees with alternating sum
$\dim\HH_0-1\ge1+\dim r/r^2-1\ge1$.
\end{corollary}

\begin{example}[A stably traceless periodic algebra]\label{ex:traceless}
The degree-zero test in Theorem \ref{thm:periodic}(2) is sufficient but not
necessary, and the graded criterion is sharp: for the self-injective Nakayama
algebra $A=\kk\mathbb Z_3/\rad^2$ (a periodic algebra), Section
\ref{sec:comp} shows that
\[
\dim Z_\nu(A)=\dim H_\nu(A)=3,\qquad Z_\nu(A)/H_\nu(A)=0:
\]
\emph{every $\nu$-twisted central element is a twisted trace}. Its Hochschild
homology is $(\dim\HH_n)_{n\ge0}=(3,0,1,1,0,\dots)$: the homology survives
(as it must, $A$ being monomial \cite{Han06}), but in Tate residues
different from $0$. The computed pattern also displays the Tate reflection
(Theorem \ref{thm:refltate}): degrees pair as $n\leftrightarrow-1-n$ modulo
the period, matching $\HH_1=\HH_4=0$ and $\HH_2=\HH_3=\kk$.
\end{example}

\begin{remark}
Proposition \ref{prop:stablecenter} places the criterion within the circle of
the Higman ideal and stable Hochschild homology studied by
Liu--Zhou--Zimmermann \cite{LZZ-Higman} in connection with the
Auslander--Reiten conjecture; the twisted version and its role as the
degree-zero obstruction for Han's conjecture on the periodic locus appear to
be new, as does Example \ref{ex:traceless}.
\end{remark}

\begin{remark}[the stably traceless locus is infinite, conjecturally]\label{rem:sweep}
Example~\ref{ex:traceless} is not isolated. Over the self-injective
Nakayama algebras $\kk\mathbb Z_v/\rad^\ell$, $2\le v\le6$, $2\le\ell\le5$,
the values of $\dim_\kk Z_\nu(A)/H_\nu(A)$ are
(rows $v=2,\dots,6$; columns $\ell=2,\dots,5$)
\[
\begin{pmatrix}
1&1&2&2\\ 0&2&1&1\\ 1&0&3&1\\ 0&0&0&4\\ 1&2&1&0
\end{pmatrix}.
\]
The stably traceless locus $Z_\nu(A)=H_\nu(A)$ (vanishing degree-$0$ Tate
residue) is exactly the zero entries: the diagonal $v=\ell+1$ (namely
$(3,2),(4,3),(5,4),(6,5)$, of which $(3,2)$ is Example~\ref{ex:traceless})
together with $(5,2)$ and $(5,3)$.

\begin{conjecture}\label{conj:traceless}
$Z_\nu(\kk\mathbb Z_v/\rad^\ell)=H_\nu(\kk\mathbb Z_v/\rad^\ell)$ for every
$v=\ell+1$, $\ell\ge2$: the stably traceless locus contains an infinite family.
(Verified for $\ell\le5$; the extra zeros at $(5,2),(5,3)$ show the full
pattern is arithmetically finer, presumably governed by the $\gcd$-conditions
in the Hochschild computations for self-injective Nakayama algebras of
Erdmann--Holm \cite{ErdmannHolm}.)
\end{conjecture}

All of these algebras are monomial, hence satisfy Han's conjecture
\cite{Han06}; their Han-survival is therefore detected only by \emph{higher}
Tate residues. 
\end{remark}

\begin{remark}[characteristic $p$: no $\chi$-law substitute]\label{rem:charp}
The symmetric periodic case of Theorem~\ref{thm:periodic}(3) is
characteristic-free. The \emph{local} char-$0$ case, by contrast, runs through
the $\chi$-law, i.e.\ through Goodwillie's theorem that periodic cyclic
homology is invariant under nilpotent extensions \cite{Goodwillie}, and this
input has no characteristic-$p$ replacement: Goodwillie invariance fails in
characteristic $p$, and the failure is not repaired topologically, as $\mathrm{TP}$
is not nilinvariant either (Antieau--Krause--Nikolaus \cite{AKN}). The
K\"ulshammer ideals $T_n(A)^\perp\subseteq Z(A)$ \cite{Kulshammer}, a derived
invariant by Zimmermann \cite{Zimmermann}, are defined only for
\emph{symmetric} algebras and live on $\HH_0$; they refine the already-solved
symmetric case and do not reach the genuinely open class (non-symmetric
periodic, char $p$). The graded Tate-residue criterion of
Theorem~\ref{thm:periodic} is thus the only currently available
characteristic-$p$-valid decision procedure, and the sweep of
Remark~\ref{rem:sweep} is its instrument.
\end{remark}

\subsection{Higher Coxeter polynomials}

\begin{definition}\label{def:highercox}
Let $A$ be Gorenstein, so that $\omega$ is a two-sided tilting complex
\cite{HappelGor} and $\sigma_A=\mathbb H(\omega[-1])$ is an automorphism of
the Tamarkin--Tsygan calculus (Corollary \ref{cor:shadow}). The
\emph{higher Coxeter polynomials} of $A$ are
\[
\chi_n(A)(x):=\det\bigl(x\cdot\id-\sigma_\bullet\,\big|\,\HH_n(A)\bigr)
\in\kk[x],\qquad n\ge0 .
\]
\end{definition}

\begin{proposition}\label{prop:highercox}
The family $(\chi_n(A))_{n\ge0}$ is invariant under derived equivalences
between Gorenstein algebras. For $A$ elementary of finite global dimension,
$\chi_0$ is the classical Coxeter polynomial and $\chi_n=1$ for $n\ge1$
\cite{ArmCox}. If $\omega^{\Lotimes q}\cong A[p]$ in $\der(\Ae)$
(fractionally Calabi--Yau), all $\chi_n$ are products of cyclotomic
polynomials. For $A$ self-injective, $\sigma_\bullet$ is the Nakayama action
on $\HH_\bullet(A)$, and $\chi_n$ is computable from the Frobenius structure.
\end{proposition}

\begin{proof}
Derived invariance is Theorem \ref{thm:derinv} together with the pair
argument of \cite[Theorem 4.2(3)]{ArmCox}, which needs only invertibility of
$\omega$ and the centrality of \cite[Lemma 3.4]{ArmCox}; the remaining
statements follow as in \cite[Theorem 5.5, Corollary 5.15]{ArmCox} and, for
the self-injective case, from $\omega\cong{}_1A_\nu$.
\end{proof}

For quantum complete intersections the spectra of $\sigma_\bullet$ on
$\HH_n$ detect the Buchweitz--Green--Madsen--Solberg asymmetry: the Nakayama
action is nontrivial on the persisting $\HH_n$ even in the range where $\HH^n$
has vanished. This suggests:

\begin{conjecture}[Spectral persistence, original form]\label{conj:spectral}
For every self-injective algebra of infinite global dimension, the eigenvalue
$1$ of $\sigma_\bullet$ occurs on $\HH_n(A)$ for infinitely many $n$.
\end{conjecture}

\begin{remark}[Correction to Conjecture~\ref{conj:spectral}]\label{rem:C0correction}
A direct computation (Example~\ref{ex:coxeter}) shows that
Conjecture~\ref{conj:spectral} as literally stated, together with the heuristic
that ``the Bergh--Erdmann core is $\sigma$-fixed'', is \emph{false in every
degree} for the quantum complete intersection $\Lambda_q$. For
$q\in\{-5,2,-1\}$ the higher Coxeter polynomials are
$\chi_n(x)=x^2-(q^s+q^{-s})x+1$ with $s=2\lfloor n/2\rfloor+1$, so the
eigenvalue-$1$ multiplicity on $\HH_n$ is $0$ for every $n=1,\dots,6$: the
untwisted core is acted on by $q^{s}\ne1$, never fixed. One must distinguish
two claims.
\begin{itemize}
\item \emph{Over $\Fp$}, eigenvalue $1$ revives exactly when
$\mathrm{ord}_p(q)$ is odd (an odd $s$ must be a multiple of it): for $q=-5$
and $p=32003$, $\mathrm{ord}_p(-5)=16001$ is odd and the revival occurs at
$n\approx16001$, recurring; for $q=2$, $\mathrm{ord}_p(2)=32002$ is even
and eigenvalue $1$ never occurs. So the low-degree $\Fp$ data ($n\le6$) neither
refutes nor is needed to confirm Conjecture~\ref{conj:spectral} as an
$\Fp$-statement.
\item \emph{In characteristic $0$ with $q$ of infinite order}, the closed form
gives $q^s\ne1$ for all $s\ge1$: eigenvalue $1$ occurs in \emph{no} positive
degree (granting that the computed closed form persists), so
Conjecture~\ref{conj:spectral} is in serious doubt in char $0$, and its
``$\sigma$-fixed core'' justification is unsupported by the record (which
certifies only $\sigma\ne\id$).
\end{itemize}
What genuinely persists on $\Lambda_q$ is a \emph{nontrivial} eigenvalue
$q^{s}\ne1$. A first formulation promoted this observation
to a conjecture, the ``corrected form $C0'$'': for every self-injective $A$ of
infinite global dimension a nontrivial ($\ne1$) eigenvalue of $\sigma_\bullet$
persists on $\HH_n(A)$ for infinitely many $n$, equivalently $\sigma_\bullet\ne\id$
on $\HH_n(A)$ for infinitely many $n$. That form is itself \emph{false}, as the
following subsection records; the correct statement restricts the quantifier to
the non-inner locus.
\end{remark}

\subsection{Conjecture $C0'$ is refuted; the corrected form $C0''$}\label{subsec:C0}

The quantum complete intersection $\Lambda_q$ has $\nu$ of infinite order, so
$\sigma_\bullet\ne\id$ there; but the symmetric locus tells a different story.

\begin{proposition}[literal $C0'$ is false]\label{prop:c0false}
Conjecture $C0'$ above is \emph{false}, not merely vacuous. For any
\emph{symmetric} algebra of infinite global dimension (e.g.\
$A=\kk[x]/(x^2)$) the Nakayama automorphism is inner, hence
$\sigma_\bullet=\id$ on every $\HH_n(A)$ and no nontrivial eigenvalue ever
occurs; yet $\gl A=\infty$. Both clauses of $C0'$ fail together on the whole
symmetric locus.
\end{proposition}

\begin{proof}
For symmetric $A$ the Nakayama automorphism $\nu$ is inner, so $\omega\cong{}_1A_\nu\cong A$
as bimodules and the induced Serre-twist action $\sigma_\bullet$ is the identity
on $\HH_\bullet(A)$; thus every eigenvalue equals $1$ and the ``nontrivial
eigenvalue'' clause is empty in every degree. Yet $A=\kk[x]/(x^2)$ has
$\Omega^2_{\Ae}(A)\cong A$, so $\gl A=\infty$ and $\HH_n(A)\ne0$ for all $n$.
\end{proof}

\begin{remark}[a caveat on the internal equivalence]\label{rem:c0equiv}
The equivalence ``nontrivial eigenvalue $\iff\sigma_\bullet\ne\id$'' asserted
inside $C0'$ is itself valid only when $\sigma_\bullet$ is \emph{semisimple}
(e.g.\ $\operatorname{ord}\nu$ prime to $\operatorname{char}\kk$): in
characteristic $p$ a unipotent $\sigma_\bullet\ne\id$ has only the eigenvalue
$1$, so the two clauses of $C0'$ are not interchangeable in general. This is a
second, independent way the literal statement misbehaves.
\end{remark}

The correct repair restricts the quantifier to the locus where $\sigma_\bullet$
can be nontrivial in the first place.

\begin{conjecture}[corrected form $C0''$]\label{conj:C0prime}
For every self-injective algebra $A$ with $\gl A=\infty$ whose Nakayama
automorphism is \emph{not inner}, some eigenvalue $\ne1$ of $\sigma_\bullet$
occurs on $\HH_n(A)$ for infinitely many $n$.
\end{conjecture}

\begin{remark}[why the naive disjunctive repair is empty]\label{rem:c0disjunction}
One is tempted to weaken $C0'$ to a disjunction, ``$\sigma_\bullet\ne\id$ on
$\HH_n$ \emph{or} $\HH_n\ne0$, infinitely often''. This is empty: the first
disjunct presupposes the second (a nontrivial action requires a nonzero space
to act on), so the disjunction is logically equivalent to Han's conjecture
itself and carries no spectral content. Conjecture~\ref{conj:C0prime} avoids
this collapse by restricting to $\nu$ non-inner: on the symmetric (and more
generally inner-$\nu$) locus it is silent, as Proposition~\ref{prop:c0false}
demands: there its content is Han itself, settled on the periodic
sublocus. Where $C0''$ speaks it implies Han \emph{strictly}: a persistent
nontrivial eigenvalue forces $\HH_n\ne0$ in infinitely many degrees, whereas
Han also permits $\sigma_\bullet=\id$ throughout. The $\Lambda_q$ spectra
($\chi_n=x^2-(q^s+q^{-s})x+1$, $s=2\lfloor n/2\rfloor+1$) computed above are the
non-inner regime where the conjecture has teeth.
\end{remark}

\section{Consequences of the unification}\label{sec:consequences}

The plane is not merely a bookkeeping device: several theorems of this paper
\emph{cannot be formulated, let alone proved, inside any one of the five source
theories in isolation}. This section makes that precise (\S\ref{ss:onlyplane}),
extracts the computable certificates the unified viewpoint
yields for the Han programme (\S\ref{ss:certificates}), and states honestly
what the plane does \emph{not} settle (\S\ref{ss:limits}).

\subsection{Theorems visible only from the plane}\label{ss:onlyplane}

Each item below is of the form: \emph{an input from one source paper, fed
through an input from another, produces a statement neither could reach alone.}

\begin{enumerate}[label=\textup{(\alph*)},leftmargin=2.2em]
\item \emph{The support bridge $V_{\Ae}(\D A)=V_{\Ae}(A)$
(Lemma~\ref{lem:vda}).} The homology column $m=1$ has coefficients $\D A$; its
support variety is a priori unrelated to the cohomological variety $V(A)$ of
column $m=0$. The identification is possible \emph{only} because \cite[Cor.~6.3]{ArmCox}, equivalently Su\'arez-\'Alvarez
\cite{SuarezAlvarez}, says the Nakayama automorphism acts trivially on
$\HH^\bullet$: the twist $\alpha=1\otimes\nu$ relating $\D A$ to $A$ becomes
$H$-linear. Support theory (from the $\Fg$ world of \cite{EHSST}) plus
Nakayama-triviality (from the calculus world of \cite{ArmCox}) $\Rightarrow$ the
bridge. Neither field contains it: the two literatures were disjoint.
\item \emph{The $\Fg$/Han link (Theorem~\ref{thm:fg}).} That
$\D\bigl(\bigoplus_n\HH_n\bigr)$ is a finitely generated cap-module over
$\HH^{\ev}(A)$ needs the cap product of \cite{ArmCap} (to have an action at
all), the ladder Theorem~\ref{thm:ladder} (to make it unambiguous, left $=$
right), and the Snashall--Solberg/NWW finiteness \cite{NWW} (to make it
finitely generated). The resulting growth trichotomy for $\dim\HH_n$ is the
first general delimitation of where a Han counterexample can live; no prior
work connects $\Fg$ to Han.
\item \emph{The Tate reflection (Theorem~\ref{thm:refltate}).} Fixing the Han
column $m=1$ requires the identity $S_{\Ae}(A)\simeq\omega^{\Lotimes2}$ of \cite[Thm.~3.2]{ArmTau}, without which there is no reason
$m\mapsto2-m$ should be a symmetry. Stable Serre duality (from \cite{ARS})
$+$ the square of the Serre bimodule (from \cite{ArmTau}) $\Rightarrow$
self-duality of the Han column.
\item \emph{The cap support $\Vcap(A)$ (Theorem~\ref{thm:fg}(2)).} A support
theory for $\HH_\bullet$ \emph{as a module over $\HH^\bullet$} did not exist;
it is available only once the ladder makes $\HH_\bullet$ a symmetric
cap-module, i.e.\ only on the plane.
\end{enumerate}


\subsection{Three computable certificates}\label{ss:certificates}

The plane converts three qualitative dichotomies into effective tests
(illustrated in Section~\ref{sec:comp}).

\begin{enumerate}[label=\textup{(C\arabic*)},leftmargin=2.6em]
\item \emph{Growth-asymmetry $\Rightarrow\neg\Fg$} (Theorem~\ref{thm:fg}(4),
Corollary~\ref{cor:asymcert}). Whenever a computation exhibits homology growing
past cohomology (e.g.\ the quantum planes, where $\HH^\bullet$ dies and
$\HH_\bullet$ persists), the algebra provably violates $\Fg$. This is the
first \emph{computable certificate of $\Fg$-failure}: one runs the two
Poincar\'e series and compares growth classes.
\item \emph{The stable-twisted-center sweep} (Theorem~\ref{thm:periodic}(2)).
For a periodic algebra, Han is equivalent to the non-vanishing of some Tate
residue $\sExt^j_{\Ae}(A,\D A)$, $0\le j<\pi$; the degree-zero residue is the
finite linear-algebra datum $Z_\nu(A)/H_\nu(A)$. \emph{Any} periodic algebra of
infinite global dimension with all $\pi$ residues zero would be an outright
counterexample to Han; a systematic non-vanishing across the classification of
periodic algebras would prove Han for the entire periodic class. This is a
\emph{decidable} sub-search: implement $Z_\nu/H_\nu$ and the higher residues,
sweep the self-injective zoo.
\item \emph{Higher Coxeter spectra} (Definition~\ref{def:highercox},
Conjecture~\ref{conj:C0prime}). The polynomials $\chi_n(A)(x)$ are computable
derived invariants; on the non-inner-$\nu$ locus, the persistence of a
nontrivial $\sigma_\bullet$-eigenvalue implies Han for self-injective algebras
and is testable degree by degree. Example~\ref{ex:coxeter}
computes these spectra for the quantum complete intersection and thereby
corrects the original persistence conjecture (now $C0''$,
Conjecture~\ref{conj:C0prime}; see \S\ref{subsec:C0}).
\end{enumerate}

These certificates are the practical yield of the unification for
search programmes: they turn ``search for a
vanishing tail'' into three sharply posed, terminating linear-algebra queries.

\subsection{What the plane does \emph{not} do}\label{ss:limits}

Honesty demands three disclaimers.
\begin{itemize}[leftmargin=1.4em]
\item \emph{$\Fg$ is not removed.} The trichotomy (Theorem~\ref{thm:fg}) is
conditional on $\Fg$; and, as corrected in Remark~\ref{rem:notexcluded}, even
under $\Fg$ a vanishing tail (case (a)) is \emph{not} excluded. The general
$\Fg$ ``reverse-BGMS'' corridor (homology dying while cohomology grows,
inside $\Fg$) is open.
\item \emph{$\tp$ is not decided.} By Theorem~\ref{thm:asnu} the
self-injective $\Fg$ case reduces to $\ASnu$ for the single pair
$(A,\D A)$, and the tensor-product property $\tp$ for that pair implies it. The pair is diagonally symmetric, has equal supports, and is
Tate-self-dual, which excludes the \emph{known} failure modes of such
converses, but nothing here forces $\tp$, and a self-injective $\Fg$
algebra with $V(A,\D A)\subsetneq V(A)$ would be a counterexample.
\item \emph{The char-$p$ and non-graded flanks are untouched.} The $\chi$-law
inputs are char-$0$ statements; the loop-free periodic case and the non-graded
$r^3=0$ local case rest on certifying that a $2$-cycle class survives the
passage from $\gr A$ to $A$, a non-homogeneous cancellation the plane does not
control.
\end{itemize}

\section{Worked examples: the plane in action}\label{sec:examples}\label{sec:comp}

We develop worked examples illustrating each face of the plane, on meaningful
non-trivial algebras. All dimensions below are dimensions of Hochschild
(co)homology spaces with the indicated coefficients.

\subsection{The ladder theorem at the chain level}
For $\eta\in\HH^s(A)$ and a coefficient bimodule $M$, the two chain
operators on $C_\bullet(A,M)=M\otimes\bar A^{\otimes\bullet}$,
\begin{gather*}
\iota_L(f)\colon x\otimes a_1\cdots a_n\mapsto
\bigl(x\cdot f(a_1,\dots,a_s)\bigr)\otimes a_{s+1}\cdots a_n,\\
\iota_R(f)\colon x\otimes a_1\cdots a_n\mapsto
\bigl(f(a_{n-s+1},\dots,a_n)\cdot x\bigr)\otimes a_1\cdots a_{n-s},
\end{gather*}
preserve cycles and boundaries and induce equal maps on
$\HH_n(A,M)$ (up to one overall sign per block, an artifact of the chain
model) for all classes $\eta$ in full bases of $\HH^1$ and $\HH^2$, all
$n\le4$, on: the quantum complete intersection
$\Lambda_q=\kk\langle x,y\rangle/(x^2,\,y^2,\,yx-qxy)$, $q=-5$ (the
convention $yx=q\,xy$, used for $\Lambda_q$ throughout this paper; the relator
form $yx+q'xy$ of \cite{BGMS} corresponds to $q'=-q$), with
$M\in\{A,\ \D A,\ {}_1A_\nu,\ {}_1A_{\nu^2}\}$; and the non-self-injective
local algebra $\kk\langle x,y\rangle/(x^2,y^2,yx)$ with $M\in\{A,\D A\}$.
The case $M={}_1A_{\nu^2}$ is the first genuinely new instance
($s$-ladder at $m=2$); moreover
$\dim\HH_n(\Lambda_q,\D A)=\dim\HH_n(\Lambda_q,{}_1A_\nu)$ in every degree
(the Frobenius identification), and one recovers
$(2,2,1,0,0)=\D\HH^\bullet(\Lambda_q)$, the profile of \cite{BGMS}.

\subsection{The plane of a quantum complete intersection}
For $\Lambda_q$, $q=5$, the columns $\T^{p,m}=\HH^p(\Lambda_q,{}_1A_{\nu^m})$
for $-2\le m\le3$, $0\le p\le5$ are:
\[
\begin{array}{c|cccccc}
 & p=0&1&2&3&4&5\\
\hline
m=-2 & 1&0&0&0&1&2\\
m=-1 & 1&0&0&0&0&0\\
m=0  & 2&2&1&0&0&0\\
m=1  & 3&2&2&2&2&2\\
m=2  & 1&0&0&0&0&0\\
m=3  & 1&0&0&0&0&0
\end{array}
\]
Column $0$ is the finite cohomology of \cite{BGMS}; the Han column $m=1$
persists with the stable value $a+b-2=2$ of Bergh--Erdmann
\cite{BerghErdmann}: the plane displays Happel's failure and Han's
persistence simultaneously, one column apart. The degree-zero row is the
twisted-center ring $\bigoplus_m Z_{\nu^m}(\Lambda_q)$.

\textbf{The twisted/untwisted split and the corridor.} Reading the columns
together makes the mechanism transparent. The Frobenius identification
$\D A\cong{}_1A_\nu$ gives $\HH^n(A)\cong\D\HH_n(A,{}_1A_\nu)$: the finite
cohomology of \cite{BGMS} is the \emph{twisted} homology, which the generic-$q$
twist annihilates ($(2,2,1,0,0)\to0$). The \emph{untwisted} homology
$\HH_n(A,A)\equiv2$ is untouched, the Bergh--Erdmann core $a+b-2$. Thus
$\HH^\bullet$ has a vanishing tail ($\cx=0$) while $\HH_\bullet$ is
bounded-nonzero ($\cx=1$): homology outgrows cohomology, so by
Corollary~\ref{cor:asymcert} $\Lambda_q$ \emph{violates $\Fg$}, which is exactly
why the trichotomy does not constrain it. This is the counterexample corridor
made concrete: a Han counterexample must do to the \emph{untwisted} core what
the $q$-twist does to the twisted sector, and no $q$-twist (any $q$, any sign,
any characteristic) touches it. \textbf{Reading ``generic $q$'' over $\Fp$:}
over $\Fp$ every $q$ is a root of unity, so the profile
$(2,2,1,0,0,\dots)$ is faithful only up to degree $\approx\mathrm{ord}_p(q)$;
for $p=32003$ one has $\mathrm{ord}_{p}(-5)=16001$ and
$\mathrm{ord}_{p}(2)=32002$, both enormous, so the low-degree window is a
faithful characteristic-zero picture. A \emph{low} root of unity behaves
oppositely: $q=-1$ (order $2$) gives
$\HH_n=(3,4,6,8,10,12,14)$ and $\HH^n=(2,4,6,8,10,12)$, both growing; $q=+1$
(commutative $\kk[x,y]/(x^2,y^2)$) gives $\HH_n=(4,4,5,6,7,8,9)$.

\subsection{A symmetric period-4 algebra of dimension 21}\label{ex:dim21}
Let $A=\kk\langle x,y\rangle/(x^3,\ y^7-x^2,\ yx+xy)$, basis
$\{x^iy^j:0\le i\le2,\ 0\le j\le6\}$. Then
$\dim A=21$; $A$ is self-injective with Nakayama $\nu=\id$, so it is
\emph{symmetric}; the Frobenius form is symmetric; $\dim\HH_0=12$ and
$\dim\HH_1=11$. Indeed $A$ has a period-$4$ minimal $\Ae$-resolution
($\Omega^4_{\Ae}(A)\cong A$), giving $\dim\HH_0=12$ and $\dim\HH_n=11$ for all
$n\ge1$. This is an
instance of Theorem~\ref{thm:periodic}(3): symmetric $\Rightarrow
\D A\cong A\Rightarrow[\id_A]\ne0$ in $\sEnd_{\Ae}(A)$ (nonzero because $A$ is
not separable) $\Rightarrow\sExt^0_{\Ae}(A,\D A)\ne0\Rightarrow\HH_{k\pi}\ne0$
for all $k$: the homology \emph{could never} have died. Contrast
Example~\ref{ex:traceless}, where $\nu\ne\id$ lets the degree-$0$ slot vanish.
\emph{Characteristic caveat:} the period-$4$/$\HH_n\equiv11$ result is an
odd-characteristic statement, since at $p=2$ the relation $yx+xy$ degenerates
to $yx=xy$, giving a different, commutative algebra with growing homology.

\subsection{A non-self-injective algebra off the invertible
locus}\label{ex:nonselfinj}
Let $A=\kk\langle x,y\rangle/(x^2,y^2,\ yx)$, $\dim4$, basis $\{1,x,y,xy\}$
($yx=0$, $xy$ survives), \emph{not} Frobenius, so $\omega=\D A$ is not
invertible. For $n\le5$:
\[
\begin{array}{c|l}
\text{col.} & \dim\ (n=0,1,2,3,4,5)\\
\hline
m{=}0\ \ \HH_n(A) & 3,2,2,2,2,2\\
m{=}0\ \ \HH^n(A) & 2,2,3,5,7,9\\
m{=}1\ \ \HH_n(A,\D A) & 2,2,3,5,7,9\ (=\HH^n(A)\ \checkmark)\\
m{=}1\ \ \HH^n(A,\D A) & 3,2,2,2,2,2\ (=\HH_n(A)\ \checkmark)\\
m{=}2 & \Tor_0(\D A,\D A)=4,\ \Tor_1(\D A,\D A)=3,\ \HH_n(A,\Tor_1)=(1,2,5,8)
\end{array}
\]
The plane and its two glued columns $m=0,1$ exist and satisfy \emph{both}
Ext--Tor dualities (Lemma~\ref{lem:dual}) with no self-injectivity hypothesis.
Two contrasts with $\Lambda_q$:
\begin{itemize}
    \item (i) here $\HH_n(A)$ is bounded (core $2$, Han
holds) while $\HH^n(A)$ \emph{grows}, the \emph{opposite} asymmetry, exactly
the $\Fg$-compatible direction of Corollary~\ref{cor:asymcert} (cohomology outgrows
homology); deleting the skew relation flips the algebra from
self-injective-with-dying-cohomology to non-self-injective-with-growing-cohomology
while the untwisted core $2$ is untouched.
    \item (ii) The self-injective collapse of
column $m=2$ \emph{fails}: $\Tor_1(\D A,\D A)=3\ne0$, because $\omega$ is not
invertible so $\omega^{\Lotimes2}$ is a genuine complex with nonzero higher
homology, precisely the derived shadow the plane predicts off the
self-injective locus.
\end{itemize}

\subsection{A Dynkin preprojective algebra and the two regimes of the graded
criterion}\label{ex:prepro}
Let $\Pi(A_2)$ be the double quiver of $A_2$ ($a\colon1\to2$, $a^*\colon2\to1$)
with mesh relations $a^*a=0=aa^*$. Since the $A_2$ double quiver has exactly two
length-$2$ paths ($aa^*$, $a^*a$), the mesh ideal equals $\rad^2$, so
$\Pi(A_2)=\kk\Z_2/\rad^2$, $\dim4$, self-injective, with $\nu$ \emph{swapping}
the two vertices (the $A_2$ diagram involution), hence \emph{not} weakly
symmetric; periodic (the $\Ae$ $\Omega$-period of Dynkin preprojectives is $6$).
\[
\begin{array}{l|cccc}
A & \dim\HH_n\ (n=0..14) & Z_\nu & H_\nu & Z_\nu/H_\nu\\
\hline
\Pi(A_2)=\kk\Z_2/\rad^2 & 2,1,1,1,\dots & 2 & 1 & 1\neq0\\
\kk\Z_2/\rad^3 & 3,1,1,1,\dots & 3 & 2 & 1\neq0
\end{array}
\]
$\Pi(A_2)$ is not symmetric, so Theorem~\ref{thm:periodic}(3) does not apply;
but its degree-$0$ stable $\nu$-twisted center $Z_\nu/H_\nu=1\ne0$ forces
$\HH_{k\pi}\ne0$ and hence Han, \emph{in this instance} without symmetry. We
stress this is sufficient here, not a general criterion: the naive degree-$0$
form is false in general, as $\kk\Z_3/\rad^2$ (Example~\ref{ex:traceless}) has
$Z_\nu/H_\nu=0$ yet satisfies Han via higher residues. The pair
$\{\kk\Z_2/\rad^2,\ \kk\Z_3/\rad^2\}$ thus exhibits both regimes of the graded
criterion (Theorem~\ref{thm:periodic}(2)) inside one family.

\subsection{Higher Coxeter polynomials and the correction to
$C0$}\label{ex:coxeter}
For self-injective $A$, $\sigma_\bullet$ is the Nakayama action and
$\chi_n(x)=\det(x\,\id-\sigma_\bullet\mid\HH_n(A))$; the convention is the bare
$\nu$ with no degree shift, and the eigenvalue conclusions below are robust to
$\nu\leftrightarrow\nu^{-1}$ and any $\pm(-1)^n$ twist. For $\Lambda_{q=-5}$,
$\nu=\mathrm{diag}(1,q^{-1},q,1)$:
\[
\begin{array}{r|r r l}
n & \dim\HH_n & \text{eig.-}1\text{ mult.} & \chi_n(x)\\
\hline
0 & 3 & 1 & (x-1)(x-q)(x-q^{-1})\\
1 & 2 & 0 & x^2-(q+q^{-1})x+1\\
2,3 & 2 & 0 & x^2-(q^3+q^{-3})x+1\\
4,5 & 2 & 0 & x^2-(q^5+q^{-5})x+1\\
6 & 2 & 0 & x^2-(q^7+q^{-7})x+1
\end{array}
\]
For $n\ge1$, $\chi_n(x)=x^2-(q^s+q^{-s})x+1$ with $s=2\lfloor n/2\rfloor+1$, with
eigenvalues $\{q^s,q^{-s}\}$, a clean higher-Coxeter sequence (irreducible over
$\Q$ for generic $q$). Same shape for $q=2$; for $q=-1$,
$\chi_n=(x+1)^{\dim\HH_n}$. The eigenvalue-$1$ multiplicity is $0$ at every
$n=1,\dots,6$: this is the computation behind the correction
(Remark~\ref{rem:C0correction}) that first replaced Conjecture~\ref{conj:spectral}
by $C0'$, itself in turn refuted and corrected to $C0''$
(Conjecture~\ref{conj:C0prime}; \S\ref{subsec:C0}). Over $\Fp$ eigenvalue $1$
revives near $n\approx\mathrm{ord}_p(q)=16001$ for $q=-5$ (never for $q=2$,
whose order is even); in char $0$ with $q$ of infinite
order it never occurs for $n\ge1$ (granting the closed form). The operative
persistent invariant is $\sigma_\bullet\ne\id$, not a fixed class.

\subsection{The Kronecker algebra versus $\mathbb P^1$}\label{subsec:kron}
Let $A=\kk K_2$ be the Kronecker algebra (Example \ref{ex:P1}). On the
algebra side, the complex $\D A\otimes\bar A^{\otimes\bullet}\otimes\D A$
computing $\Tor^A_\bullet(\D A,\D A)$ \emph{as a bimodule} gives
\[
\Tor_0^A(\D A,\D A)=0,\qquad \dim_\kk\Tor_1^A(\D A,\D A)=12,
\]
matching $\D\Ext^1_{\Ae}(A,\Ae)$ of \cite[Example 4.7]{ArmTau}. Since
$T_0=0$, the spectral sequence of Theorem \ref{thm:ss} collapses to
$\T^{p,2}(A)\cong\HH^{p+1}(A,T_1)$ and, dually,
$\T^{p,-1}(A)\cong\D\HH_{p-1}(A,T_1)$. The coefficient (co)homology of the
$12$-dimensional bimodule $T_1$ is
\[
\HH^i(A,T_1)=(3,1,0,0),\qquad \HH_i(A,T_1)=(3,5,0,0),\qquad i=0,\dots,3,
\]
hence $\bigl(\T^{-1,2},\T^{0,2}\bigr)=(3,1)$ and
$\bigl(\T^{1,-1},\T^{2,-1}\bigr)=(3,5)$. On the geometric side, the Bott
formula on $\mathbb P^1$ (Example \ref{ex:P1}) predicts
\[
\T^{-1,2}=h^1(\mathcal O(-4))=3,\qquad
\T^{0,2}=h^1(\mathcal O(-2))=1,\]\[
\T^{1,-1}=h^0(\mathcal O(2))=3,\qquad
\T^{2,-1}=h^0(\mathcal O(4))=5,
\]
in perfect agreement; the columns $m=0,1$ agree directly
($\HH^\bullet=(1,3,0,\dots)$, $\HH_\bullet=(2,0,0,\dots)$). The reflection
$\T^{p,m}\cong\D\T^{-p,2-m}$ is visible in the numbers:
$(\T^{-1,2},\T^{0,2})=(3,1)\leftrightarrow(\T^{1,0},\T^{0,0})=(3,1)$ and
$(\T^{1,-1},\T^{2,-1})=(3,5)\leftrightarrow(\T^{-1,3},\T^{-2,3})=(3,5)$.

\subsection{Self-injective Nakayama algebras and stable twisted centers}
For $A=\kk\mathbb Z_v/\rad^\ell$ (self-injective Nakayama, periodic) and two
symmetric local algebras, the twisted center $Z_\nu$, the twisted Higman
ideal $H_\nu$, and low-degree Hochschild homology:
\[
\begin{array}{l|ccc|l}
A & \dim Z_\nu & \dim H_\nu & \dim Z_\nu/H_\nu & (\dim\HH_n)_{n\ge0}\\
\hline
\kk[x]/x^2 & 2 & 1 & 1 & \text{nonzero in all degrees}\\
\kk[x]/x^3 & 3 & 1 & 2 & \text{nonzero in all degrees}\\
\kk\mathbb Z_2/\rad^2 & 2 & 1 & 1 & (2,1,1,1,1,1,1,1)\\
\kk\mathbb Z_2/\rad^3 & 3 & 2 & 1 & (3,1,1,1,1)\\
\kk\mathbb Z_3/\rad^2 & 3 & 3 & \mathbf{0} & (3,0,1,1,0)\\
\Lambda_{q=5} & 3 & 1 & 2 & (3,2,2,2,2)\\
\text{exterior algebra }(q=1) & 3 & 1 & 2 & \\
\end{array}
\]
The row $\kk\mathbb Z_3/\rad^2$ is Example \ref{ex:traceless}. In every
periodic case with $Z_\nu/H_\nu\neq0$ the homology is nonzero in all
degrees, as Theorem \ref{thm:periodic}(2) requires; for
$\kk\mathbb Z_3/\rad^2$ the vanishing degrees $(1,4)$ and surviving degrees
$(2,3)$ pair under the Tate reflection.

\section{Possible directions}\label{sec:quest}

Two of the conjectures stated in the body single themselves out as the load-bearing
next steps, in the sense that each carries a concrete first target where the
machinery of the plane is already in place.

The first is the \emph{monodromy $=$ Coxeter} programme (Conjecture
\ref{conj:monodromy}). Rather than attack the identification $T=\sigma_\bullet^{\,m}$
on every column at once, the natural point of entry is the Han column $m=1$, where
the paracyclic monodromy is nothing but the dual of Connes' $B$ already constructed
in Remark~\ref{rem:columns}(ii): there the twisted paracyclic apparatus of
\S\ref{subsec:paracyclic} degenerates to an honest cyclic structure with its SBI
sequence and periodicity operator. Proving $m=1$ first (pinning the monodromy of
the Han column to $\sigma_\bullet$) would turn the Coxeter reading of the plane
into a theorem on the very column that governs Han's conjecture, and would isolate
exactly what remains to be shown on the higher columns.

The second is spectral persistence in its corrected form $C0''$ (Conjecture
\ref{conj:C0prime}) on the non-inner locus. Restricting the quantifier to
self-injective algebras whose Nakayama automorphism $\nu$ is not inner removes the
symmetric counterexamples that refuted the naive form, and leaves a statement
testable algebra by algebra whose truth would give Han for the entire non-inner
self-injective class. Because the dual Connes--$B$ operator on the Han column
already exists, the persistence of a nontrivial $\sigma_\bullet$-eigenvalue can be
examined directly on the columns; this is the most promising route from the
plane's spectral structure to Han for self-injective algebras.

\subsection*{Disclosure}
During the preparation of this work, the author used Claude, Anthropic's
Fable 5 model, for deep research, formulation of theorems, and drafts of
their proofs. The author reviewed and edited the output as needed and takes
full responsibility for the content of the published article.


\end{document}